\documentclass[a4paper]{amsart}
\usepackage[utf8]{inputenc}
\usepackage[T1]{fontenc}
\usepackage{lmodern}
\usepackage{amssymb}
\usepackage[all]{xy}
\usepackage{nicefrac,mathtools,enumitem}
\usepackage{microtype}
\usepackage[pdftitle={Non-Hausdorff Symmetries of C*-algebras},
  pdfauthor={Alcides Buss, Ralf Meyer, Chenchang Zhu},
  pdfsubject={Mathematics}
]{hyperref}
\usepackage[lite]{amsrefs}
\newcommand*{\MRref}[2]{ \href{http://www.ams.org/mathscinet-getitem?mr=#1}{MR #1}}
\newcommand*{\arxiv}[1]{\href{http://www.arxiv.org/abs/#1}{arXiv: #1}}

\numberwithin{equation}{section}
\theoremstyle{plain}
\newtheorem{theorem}[equation]{Theorem}
\newtheorem{lemma}[equation]{Lemma}
\newtheorem{proposition}[equation]{Proposition}
\newtheorem{corollary}[equation]{Corollary}
\theoremstyle{definition}
\newtheorem{definition}[equation]{Definition}
\theoremstyle{remark}
\newtheorem{remark}[equation]{Remark}
\newtheorem{example}[equation]{Example}

\DeclareMathOperator{\source}{s}% source of groupoids
\DeclareMathOperator{\target}{r}% target of groupoids
\DeclareMathOperator{\Aut}{Aut}% automorphism group
\DeclareMathOperator{\Inn}{Inn}% inner automorphism group
\DeclareMathOperator{\Out}{Out}% outer automorphism group

\newcommand*{\Rotc}[1]{A_#1}%rotation C*-algebra with parameter
\newcommand*{\nb}{\nobreakdash}
\newcommand*{\Star}{$^*$\nobreakdash-}

\newcommand*{\C}{\mathbb C}
\newcommand*{\Z}{\mathbb Z}
\newcommand*{\R}{\mathbb R}
\newcommand*{\N}{\mathbb N}

\newcommand*{\Torus}{\mathbb T}%torus
\newcommand*{\Bound}{\mathbb B}%adjointable operators on a Hilbert module
\newcommand*{\Comp}{\mathbb K}%compact operators on a Hilbert module
\newcommand*{\Mat}{\mathbb M}%matrices

\newcommand*{\Cst}{\textup C^*}%C*-algebra
\newcommand*{\Cont}{\textup C}%continuous functions
\newcommand*{\Mult}{\mathcal M}%multiplier algebra

\newcommand*{\K}{\textup K}%K-theory

\newcommand*{\ima}{\textup i}%imaginary unit
\newcommand*{\Eul}{\textup e}%Euler number
\newcommand*{\Id}{\textup{Id}}%identity
\newcommand*{\Ad}{\textup{Ad}}%conjugation by a unitary
\newcommand*{\leftsub}{\textup l}%left for subscripts
\newcommand*{\rightsub}{\textup r}%right for subscripts

\newcommand*{\Cat}{\mathcal C}%category
\newcommand*{\cC}{\mathcal C}%category
\newcommand*{\Cattwo}{\mathcal C}%2-category
\newcommand*{\Hils}{\mathcal H}%Hilbert space
\newcommand*{\Hilm}[1][H]{\mathcal #1}%Hilbert module

\newcommand*{\U}{\mathcal U}% unitaries
\newcommand*{\A}{\mathcal A}% C*-bundles
\newcommand*{\actA}{\alpha}% action on a C*-algebra/bundle A
\newcommand*{\acm}{c}%conjugation action of a crossed module
\newcommand*{\tcm}{\partial}%target map of a crossed module
\newcommand*{\twc}{v}% twist cocycle of weak crossed module action on C*-algebra A

\newcommand*{\horizprod}{\mathbin{\cdot_\textup h}}%horizontal multiplication
\newcommand*{\vertprod}{\mathbin{\cdot_\textup v}}%vertical multiplication

\newcommand*{\defeq}{\mathrel{\vcentcolon=}}
\newcommand*{\congto}{\xrightarrow\sim}

\newcommand*{\norm}[1]{\lVert#1\rVert}%norm
\newcommand*{\braket}[2]{\left\langle#1\!\mid\!#2\right\rangle}%inner products

\newcommand*{\Csttwocat}{\mathfrak C^*(2)}%2-category of C*-algebras
\newcommand*{\Corrcat}{\mathfrak{Corr}(2)}%2-category of C*-algebras with correspondences

\begin{document}
\title[A higher category approach to twisted actions on $\Cst$-algebras]{A higher category approach to\\twisted actions on $\Cst$-algebras}

\author{Alcides Buss}
\email{alcides@mtm.ufsc.br}
\address{Departamento de Matem\'atica\\
 Universidade Federal de Santa Catarina\\
 88.040-900 Florian\'opolis-SC\\
 Brazil}

\author{Ralf Meyer}
\email{rameyer@uni-math.gwdg.de}
\address{Mathematisches Institut and Courant Research Centre ``Higher Order Structures''\\
 Georg-August-Universit\"at G\"ottingen\\
 Bunsenstra\ss e 3--5\\
 37073 G\"ottingen\\
 Germany}

\author{Chenchang Zhu}
\email{zhu@uni-math.gwdg.de}
\address{Courant Research Centre ``Higher Order Structures''\\
 Georg-August-Universit\"at G\"ottingen\\
 Bunsenstra\ss e 3--5\\
 37073 G\"ottingen\\
 Germany}

\begin{abstract}
  \(\Cst\)\nb-algebras form a \(2\)\nb-category with \Star{}homomorphisms or correspondences as morphisms and unitary intertwiners as \(2\)\nb-morphisms.  We use this structure to define weak actions of \(2\)\nb-categories, weakly equivariant maps between weak actions, and modifications between weakly equivariant maps.  In the group case, we identify the resulting notions with known ones, including Busby--Smith twisted actions and equivalence of such actions, covariant representations, and saturated Fell bundles.  For \(2\)\nb-groups, weak actions combine twists in the sense of Green and Busby--Smith.

  The Packer--Raeburn Stabilisation Trick implies that all Busby--Smith twisted group actions of locally compact groups are Morita equivalent to classical group actions.  We generalise this to actions of strict \(2\)\nb-groupoids.
\end{abstract}
\subjclass[2000]{46L55, 18D05}
\keywords{Green twisted action, Busby--Smith twisted action, Fell bundle, bicategory, \(2\)\nb-category, correspondence, Morita equivalence, Packer--Raeburn Stabilisation Trick}
\thanks{Supported by the German Research Foundation (Deutsche Forschungsgemeinschaft (DFG)) through the Institutional Strategy of the University of G\"ottingen.}
\maketitle

\tableofcontents

\section{Introduction}
\label{sec:introduction}

An automorphism of a \(\Cst\)\nb-algebra~\(A\) is called \emph{inner} if it is of the form \(\Ad_u\colon a\mapsto uau^*\) for some unitary multiplier~\(u\) of~\(A\).  Inner automorphisms act trivially on \(\K\)\nb-theory and all other interesting invariants for \(\Cst\)\nb-algebras; more precisely, they act trivially on a functor~\(F\) if the corner embedding \(F(A)\to F\bigl(\Mat_2(A)\bigr)\) is invertible for all~\(A\).  If two automorphisms \(\alpha_1\) and~\(\alpha_2\) of a \(\Cst\)\nb-algebra~\(A\) differ by an inner automorphism, \(\alpha_2=\Ad_u\circ\alpha_1\), then their crossed product \(\Cst\)\nb-algebras are isomorphic.

While these statements suggest that we may simply ignore inner automorphisms, this is false for representations of more general groups.  For instance, any automorphism of the \(\Cst\)\nb-algebra~\(\Comp\) of compact operators on a separable Hilbert space is inner, so that any two group actions on~\(\Comp\) agree up to inner automorphisms.  Group actions by \Star{}automorphisms on~\(\Comp(\Hils)\) correspond to projective representations of the group on the underlying Hilbert space~\(\Hils\).  But the stabilised rotation algebras \(\Rotc\vartheta\otimes\Comp\) for \(\vartheta\in[0,1]\) can be realised as crossed products with \(\Comp(\Hils)\) for a family of projective representations of~\(\Z^2\) on~\(\Hils\).  Hence there are actions of~\(\Z^2\) on~\(\Comp(\Hils)\) for which the crossed products are not Morita equivalent, although all automorphisms of \(\Comp(\Hils)\) are inner.  The correct way to formulate the insignificance of inner automorphisms is the well-known \emph{exterior equivalence} of group actions.  It differs from the naive notion by a cocycle condition for the twisting unitaries.

Similarly, a group homomorphism to the \emph{outer automorphism group}
\[
\Out(A) \defeq \Aut(A)\bigm/\Inn(A)
\]
does not qualify as a group action on~\(A\) because no crossed product can be defined in such a situation (of course, \(\Aut(A)\) denotes the automorphism group of~\(A\) and \(\Inn(A)\) the normal subgroup of inner automorphisms).  The correct way to define group actions up to inner automorphisms are the \emph{twisted actions} in the sense of Busby and Smith~\cite{Busby-Smith:Representations_twisted_group}.

Non-Hausdorff symmetry groups of \(\Cst\)\nb-algebras are related to inner automorphisms in~\cite{Buss-Meyer-Zhu:Non-Hausdorff_symmetries}. For example, the rotation algebra~\(\Rotc\vartheta\) plays the role of the algebra of functions on the non-commutative space~\(\Torus/\lambda^\Z\) for \(\lambda\defeq \exp(2\pi\ima\vartheta)\).  Since this quotient group acts on itself by multiplication, we expect~\(\Rotc\vartheta\) to carry a canonical action of~\(\Torus/\lambda^\Z\).  Since~\(\lambda^\Z\) is dense in~\(\Torus\) for irrational~\(\vartheta\), such an action cannot exist in the classical sense.  Instead, we have an action of~\(\Torus\) that restricts to an inner action on~\(\lambda^\Z\), so that we get a natural homomorphism \(\Torus/\lambda^\Z \to \Out(\Rotc\vartheta)\).  We already remarked above that, even for classical groups, a homomorphism to \(\Out(\Rotc\vartheta)\) is not the right way to define group actions up to inner automorphisms.  Instead, we need a pair of group homomorphisms from~\(\Torus\) to~\(\Aut(\Rotc\vartheta)\) and from~\(\Z\) to the group of unitary multipliers~\(\U\Mult(\Rotc\vartheta)\) satisfying two compatibility conditions familiar from the twisted group actions of Philip Green~\cite{Green:Local_twisted}.

In this article, we interpret twisted group actions and related notions from the point of view of \(2\)\nb-category theory.  Weak \(2\)\nb-categories, which are also called bicategories, were used in an operator algebraic context in \cites{Brouwer:Bicategorical, Landsman:Bicategories} in order to study Morita equivalence of \(\Cst\)\nb-algebras and von Neumann algebras.  The authors implicitly used \(2\)\nb-categories in~\cite{Buss-Meyer-Zhu:Non-Hausdorff_symmetries} in order to describe non-Hausdorff symmetry groups of \(\Cst\)\nb-algebras.

In~\cite{Buss-Meyer-Zhu:Non-Hausdorff_symmetries} non-Hausdorff symmetry groups are described by crossed modules.  Crossed modules are already considered in~\cite{Raeburn-Sims-Williams:Twisted_obstructions} in order to study some obstruction problems for twisted actions cohomologically.  Since crossed modules are equivalent to strict \(2\)\nb-groups, their appearance in \cites{Buss-Meyer-Zhu:Non-Hausdorff_symmetries, Raeburn-Sims-Williams:Twisted_obstructions} is of fundamental importance.  The general theory of \(2\)\nb-categories explains various definitions related to twisted group actions, and it even provides some insights for ordinary group actions.

The theory extends to continuous actions of locally compact topological \(2\)\nb-cat\-egories as well, but we do not discuss the relevant continuity conditions in the introduction.

The first example of a strict \(2\)\nb-category is the \(2\)\nb-category of categories.  Its \(2\)\nb-category structure encodes the familiar properties of categories (objects), functors (arrows between objects), and natural transformations (arrows between arrows).  One of the basic ideas of category theory is that two functors that are related by an invertible natural transformation should be considered equivalent.  This leads to the notion of equivalence of categories.

If two groups are isomorphic, then it is usually important to remember the isomorphism between them.  Similarly, if two functors are naturally isomorphic, we should remember the natural isomorphism between them and require suitable coherence laws.  An example of this are products, say, of sets.  When we construct products explicitly using the set theory axioms, then \((X\times Y)\times Z\) and \(X\times (Y\times Z)\) are different but naturally isomorphic, and so are \(X\times Y\) and \(Y\times X\).  But the existence of such isomorphisms is not enough: we must select such natural isomorphisms and then check certain coherence laws.  Depending on the coherence laws we use, we arrive at the notion of a symmetric or a braided monoidal category (see~\cite{Baez-Dolan:Categorification}).

The idea to weaken equality of morphisms to isomorphism of morphisms works in any \(2\)\nb-category and is very familiar from homotopy theory, where continuous maps are considered equivalent if they are homotopic.  We will apply this idea to \(\Cst\)\nb-algebras.  There are several possible ways to turn \(\Cst\)\nb-algebras into a \(2\)\nb-category.  We will use two definitions here, one using \Star{}representations and unitary intertwiners, the other using correspondences and unitary intertwiners.  By definition, a \emph{\Star{}representation} from~\(A\) to~\(B\) is a non-degenerate \Star{}homomorphism from~\(A\) to the multiplier algebra of~\(B\), and a \emph{correspondence} from~\(A\) to~\(B\) is a non-degenerate \Star{}homomorphism from~\(A\) to the algebra of adjointable operators on some Hilbert \(B\)\nb-module.  A unitary intertwiner between two \Star{}representations \(f,g\colon A\rightrightarrows\Mult(B)\) is a unitary multiplier~\(u\) of~\(B\) with
\[
uf(a)=g(a)u\qquad\text{for all \(a\in A\).}
\]
Unitary intertwiners between correspondences are \(A\)\nb-linear unitary operators between the underlying Hilbert \(B\)\nb-modules.

Whereas \(\Cst\)\nb-algebras with \Star{}representations and unitary intertwiners form a strict \(2\)\nb-category, \(\Cst\)\nb-algebras with correspondences and unitary intertwiners only form a weak \(2\)\nb-category (bicategory) because the (tensor) product of correspondences is only associative up to isomorphism and only has units up to isomorphism.

Recall that categories form a \(2\)\nb-category with categories as objects, functors as \(1\)\nb-morphisms, and natural transformations as \(2\)\nb-morphisms.  Similarly, \(2\)\nb-categories form a \(3\)\nb-category.  Therefore, given a \(2\)\nb-category, or just a group viewed as a \(2\)\nb-category, there are three levels of morphisms from this \(2\)\nb-category to the two \(2\)\nb-categories of \(\Cst\)\nb-algebras mentioned above.

First, there are functors between \(2\)\nb-categories.  A strict functor from a group to the \(2\)\nb-category of \(\Cst\)\nb-algebras is a group action.  When we weaken equalities of morphisms to isomorphisms, we get \emph{Busby--Smith twisted actions}.

Secondly, there are natural transformations between such functors.  The strict natural transformations between group actions are exactly the \emph{equivariant maps}.  The weak version of this combines equivariant maps with \emph{exterior equivalence}: a weakly equivariant map between two weak actions of a group is a strictly equivariant map to an exterior equivalent weak action.  Moreover, if the target algebra carries the trivial weak group action, then a weakly equivariant map is the same thing as a \emph{covariant representation}.  Thus weak natural transformations contain equivariant maps, exterior equivalence, and covariant representations as special cases.

Thirdly, there are modifications between weakly equivariant maps.  These generalise unitary intertwining operators between covariant representations to the setting where the target algebra carries a non-trivial action and provide the correct notion of equivalence for weakly equivariant maps between weak actions.

General \(2\)\nb-category theory (see~\cite{Leinster:Basic_Bicategories}) explains the various coherence laws or covariance conditions that appear in the above definitions.  Once you understand \(2\)\nb-categories, you no longer have to memorise the cocycle condition for Busby--Smith twisted actions or the definition of exterior equivalence of such twisted actions because you can derive these conditions in a minute.

Even more, the general theory applies equally well to actions of more general \(2\)\nb-categories.  For strict \(2\)\nb-groupoids, represented by crossed modules, the notion of a strict action is a straightforward generalisation of Philip Green's twisted actions (see~\cite{Buss-Meyer-Zhu:Non-Hausdorff_symmetries}).  Weak actions of crossed modules consistently combine the twists of Green and Busby--Smith.  The relevant coherence laws are already rather complicated, but are produced automatically by higher category theory.

Using the weak \(2\)\nb-category of \(\Cst\)\nb-algebras with correspondences and unitary intertwiners, we get yet another setup to study actions of groups or \(2\)\nb-categories on \(\Cst\)\nb-algebras.  Since the target category is not strict, we only have the weak notions of action, equivariant map, and equivalence here.  We show that weak actions of a group are equivalent to \emph{saturated Fell bundles} over this group.  We interpret weakly equivariant maps as correspondences of Fell bundles.  Modifications are isomorphisms of correspondences.

The interesting feature of actions by correspondences is that we can transport them under Morita equivalences.  That is, if~\(G\) acts by correspondences on~\(A\) and~\(A\) is Morita equivalent to~\(B\), then we get an induced action of~\(G\) by correspondences on~\(B\).  Actions of~\(\Z\) of this kind were already studied in~\cite{Abadie-Eilers-Exel:Morita_bimodules} -- free groups are the only case where we do not need higher categories to define weak actions.

Recall that a Morita equivalence between two \(\sigma\)\nb-unital \(\Cst\)\nb-algebras implies a \Star{}isomorphism between their \(\Cst\)\nb-stabilisations.  A more precise version of this statement shows that a weak action by correspondences of a strict \(2\)\nb-groupoid on a \(\sigma\)\nb-unital \(\Cst\)\nb-algebra becomes a weak action by \Star{}isomorphisms after \(\Cst\)\nb-stabilisation.  A basic fact about Busby--Smith twisted group actions is the Packer--Raeburn Stabilisation Trick, which shows that Busby--Smith twisted actions are Morita equivalent to untwisted group actions.  Combining both results, we find that a group action by correspondences is equivalent to a classical group action on the \(\Cst\)\nb-stabilisation.  As a consequence, a group action by correspondences or, equivalently, a saturated Fell bundle with given unit fibre is exactly the structure that is inherited by \(\Cst\)\nb-algebras that are Morita equivalent to \(\Cst\)\nb-algebras with a classical group action.

The same argument applies to crossed modules of locally compact topological groupoids: any topological weak action by correspondences of such a crossed module is equivalent to a strict action of the crossed module.  Such strictification results are rather unusual in higher category theory.  For most target categories, weak actions are genuinely more general than strict actions.

\section{\texorpdfstring{$2$}{2}-categories and crossed modules}
\label{sec:two_categories}

Here we review strict and weak \(2\)\nb-categories.  Our motivating examples are the \(2\)\nb-categories of \(\Cst\)\nb-algebras \(\Csttwocat\) and \(\Corrcat\) based on \Star{}representations or correspondences, respectively, and unitary intertwiners.

\subsection{Strict \texorpdfstring{$2$}{2}-categories}

The quick definition of a strict \(2\)\nb-category describes it as a category enriched over categories.  That is, for two objects \(x\) and~\(y\) of our first order category, we have a \emph{category} of morphisms from~\(x\) to~\(y\), and the composition of morphisms lifts to a bifunctor between these morphism categories.  This definition is similar to the definition of a pre-additive category: the latter is nothing but a category enriched over Abelian groups.

We now write down more explicitly what a category enriched over categories is (see also~\cite{Baez:Introduction_n}).  Having categories of morphisms boils down to having \emph{arrows} between objects \(x\to y\), also called \(1\)\nb-morphisms, and arrows between arrows
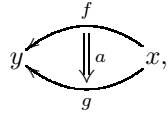
\[
\xymatrix@1@C+2em{
  y &
  \ar@/_1pc/[l]_{f}_{}="0"
  \ar@/^1pc/[l]^{g}^{}="1"
  \ar@{=>}"0";"1"^{a}
  x,
}
\]
which are called \(2\)\nb-morphisms or \emph{bigons} because of their shape.  We prefer to call them bigons because there are other ways to describe \(2\)\nb-categories that use triangles or even more complicated shapes as \(2\)\nb-morphisms (see~\cite{Baez:Introduction_n}).

The category structure on the space of arrows \(x\to y\) provides a \emph{vertical composition} of bigons
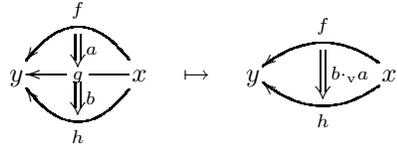
\[
\xymatrix@1@C+1.5em{
  y &
  \ar@/_1.5pc/[l]_{f}_{}="0"
  \ar[l]|{\phantom{|}g\phantom{|}}^{}="1"_{}="1b"
  \ar@/^1.5pc/[l]^{h}^{}="2"
  \ar@{=>} "0";"1"^{a}
  \ar@{=>} "1b";"2"^{b}
  x}
\quad\mapsto\quad
\xymatrix@1@C+2em{
  y & \ar@/_1pc/[l]_{f}_{}="4"
  \ar@/^1pc/[l]^{h}^{}="5"
  \ar@{=>}"4";"5"^{b\vertprod a}
  x.
}
\]
The composition functor between the arrow categories provides a composition of arrows
\[
\xymatrix@1{z & \ar[l]_{f} y & \ar[l]_{g} x} \quad\mapsto\quad
\xymatrix@1{z & \ar[l]_{f g} x}
\]
and a \emph{horizontal composition} of bigons
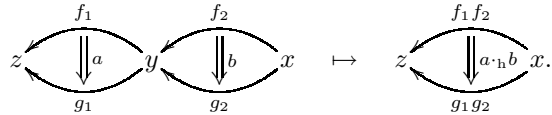
\[
\xymatrix@1@C+2em{
  z &
  \ar@/_1pc/[l]_{f_1}_{}="0"
  \ar@/^1pc/[l]^{g_1}^{}="1"
  \ar@{=>}"0";"1"^{a}
  y &
  \ar@/_1pc/[l]_{f_2}_{}="2"
  \ar@/^1pc/[l]^{g_2}^{}="3"
  \ar@{=>}"2";"3"^{b}
  x}
\quad\mapsto\quad
\xymatrix@1@C+2em{
  z & \ar@/_1pc/[l]_{f_1f_2}_{}="4"
  \ar@/^1pc/[l]^{g_1g_2}^{}="5"
  \ar@{=>}"4";"5"^{a\horizprod b}
  x.
}
\]
These three compositions of arrows and bigons are associative and unital, with the same unit bigons for the vertical and horizontal product.  Furthermore, the horizontal and vertical products commute: given a diagram
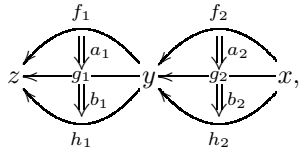
\[
\xymatrix@1@C+2em{
  z &
  \ar@/_1.5pc/[l]_{f_1}_{}="0"
  \ar[l]|{g_1}^{}="1"_{}="1b"
  \ar@/^1.5pc/[l]^{h_1}^{}="2"
  \ar@{=>} "0";"1"^{a_1}
  \ar@{=>} "1b";"2"^{b_1}
  y &
  \ar@/_1.5pc/[l]_{f_2}_{}="3"
  \ar[l]|{g_2}^{}="4"_{}="4b"
  \ar@/^1.5pc/[l]^{h_2}^{}="5"
  \ar@{=>} "3";"4"^{a_2}
  \ar@{=>} "4b";"5"^{b_2}
  x,
}
\]
composing first vertically and then horizontally or the other way around produces the same bigon \(f_1f_2\Rightarrow h_1h_2\).  The three composition operations and the associativity, unitality, and interchange conditions above only make explicit the structure present in a category enriched over categories.

In any strict \(2\)\nb-category, the objects and arrows form an ordinary category.

\begin{example}
  \label{exa:categories_twocat}
  Categories form a strict \(2\)\nb-category with small categories as objects, functors between categories as arrows, and natural transformations between functors as bigons.  The composition of arrows is the composition of functors and the vertical composition of bigons is the composition of natural transformations.  The horizontal composition of bigons yields the canonical natural transformation
  \[
  \Phi_{1,G_2(A)}\circ F_1(\Phi_{2,A}) = G_1(\Phi_{2,F_2(A)})\circ \Phi_{1,F_2(A)}\colon
  F_1\bigl(F_2(A)\bigr)\to G_1\bigl(G_2(A)\bigr)
  \]
  for the diagram
  
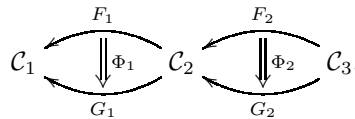
\[
  \xymatrix@C+2em{
    \Cat_1 &
    \ar@/_1pc/[l]_{F_1}_{}="0"
    \ar@/^1pc/[l]^{G_1}^{}="1"
    \ar@{=>}"0";"1"^{\Phi_1}
    \Cat_2 &
    \ar@/_1pc/[l]_{F_2}_{}="2"
    \ar@/^1pc/[l]^{G_2}^{}="3"
    \ar@{=>}"2";"3"^{\Phi_2}
    \Cat_3.
  }
  \]
\end{example}

\subsubsection{The strict \texorpdfstring{$2$}{2}-category \texorpdfstring{$\Csttwocat$}{C*(2)} of \texorpdfstring{$\Cst$}{C*}-algebras}

Now we describe a strict \(2\)\nb-category with \(\Cst\)\nb-algebras as objects, non-degenerate \Star{}homomorphisms \(A\to\Mult(B)\) as arrows, and unitary intertwiners between such \Star{}homomorphisms as bigons.

Let \(A\) and~\(B\) be \(\Cst\)\nb-algebras and let \(\Mult(A)\) and~\(\Mult(B)\) be their multiplier algebras, equipped with the strict topologies.  By definition, an arrow from~\(A\) to~\(B\) is a strictly continuous, unital \Star{}homomorphism from~\(\Mult(A)\) to~\(\Mult(B)\).  These arrows are composed in the obvious way.  Since~\(A\) is strictly dense in its multiplier algebra, an arrow is determined uniquely by its restriction to a \Star{}homomorphism from~\(A\) to~\(\Mult(B)\).  A \Star{}homomorphism from~\(A\) to~\(\Mult(B)\) extends to a strictly continuous, unital \Star{}homomorphism from~\(\Mult(A)\) to~\(\Mult(B)\) if and only if it is non-degenerate.  We also call non-degenerate \Star{}homomorphisms from~\(A\) to~\(\Mult(B)\) \emph{\Star{}representations} of~\(A\) on~\(B\).  Thus the arrows in~\(\Csttwocat\) are equivalent to \emph{\Star{}representations} of~\(A\) on~\(B\); some authors simply call them \emph{morphisms}.

Let \(f\) and~\(g\) be two arrows from~\(A\) to~\(B\), that is, strictly continuous, unital \Star{}homomorphisms from~\(\Mult(A)\) to~\(\Mult(B)\).  An element \(b\in\Mult(B)\) is called an \emph{intertwiner} from~\(f\) to~\(g\) if \(b\cdot f(a) = g(a)\cdot b\) for all \(a\in \Mult(A)\) or, equivalently, for all \(a\in A\).  If~\(b\) is unitary, this is equivalent to \(g = \Ad_b\circ f\), where~\(\Ad_b\) is the inner automorphism generated by~\(b\).  The set of bigons from~\(f\) to~\(g\) in~\(\Csttwocat\) is the set of unitary intertwiners from~\(f\) to~\(g\).  The vertical composition of intertwiners is the product in \(\Mult(B)\), the horizontal composition of two bigons \(c\colon f_1\Rightarrow g_1\) and \(b\colon f_2\Rightarrow g_2\) for composable pairs of arrows \(f_1,g_1\colon B\rightrightarrows C\) and \(f_2,g_2\colon A\rightrightarrows B\) is
\begin{equation}
  \label{eq:horizprod_Cstartwocat}
  c\horizprod b \defeq c\cdot f_1(b) = g_1(b)\cdot c.
\end{equation}
It is easy to check that the vertical and horizontal compositions of bigons are associative and satisfy the interchange law.

This strict \(2\)\nb-category~\(\Csttwocat\) may be modified in several ways.  We may allow non-unitary intertwiners.  We may also allow non-unital strictly continuous \Star{}homomorphisms \(\Mult(A)\to\Mult(B)\) as our morphisms.

\begin{remark}
  \label{rem:Cstarcat_without_multipliers}
  When we want to study \(\K\)\nb-theory, we should use possibly degenerate \Star{}homomorphisms \(A\to B\) as morphisms.  This creates problems with the horizontal composition of bigons because degenerate \Star{}homomorphisms need not act on multipliers.  This problem may be avoided if we replace \(\Mult(B)\) by the unitalisation~\(B^+\) of~\(B\).  This leads to a \(2\)\nb-category with \(\Cst\)\nb-algebras as objects, \Star{}homomorphisms \(A\to B\) as arrows from~\(A\) to~\(B\) and unitary intertwiners in~\(B^+\) as bigons.
\end{remark}

\subsubsection{Strict \texorpdfstring{$2$}{2}-groupoids}
\label{sec:strict-2-gpd}

\begin{definition}
  \label{def:strict_two-groupoid}
  A \emph{strict \(2\)\nb-groupoid} is a strict \(2\)\nb-category in which all arrows and bigons are invertible (bigons are invertible both for the vertical and horizontal product); a \emph{strict \(2\)\nb-group} is a strict \(2\)\nb-groupoid with a single object.
\end{definition}

Given a strict \(2\)\nb-groupoid, its objects and arrows form a groupoid, which we call~\(G\).  We may use horizontal products with unit bigons to produce any bigon from a bigon that starts at a unit morphism:
\[
\xymatrix@C+2em{
  y &
  \ar@/_1pc/[l]_{1_y}_{}="0"
  \ar@/^1pc/[l]^{\tcm(a)}^{}="1"
  \ar@{=>}"0";"1"^{a}
  y &
  \ar@/_1pc/[l]_{f}_{}="2"
  \ar@/^1pc/[l]^{f}^{}="3"
  \ar@{=>}"2";"3"^{1_f}
  x}
\quad\mapsto\quad
\xymatrix@C+2em{
  y & \ar@/_1pc/[l]_{f}_{}="4"
  \ar@/^1pc/[l]^{\tcm(a)f}^{}="5"
  \ar@{=>}"4";"5"^{a\horizprod f}
  x.
}
\]
The bigons starting at the identity morphism on~\(x\) form a group~\(H_x\) with respect to horizontal composition, and the map \(h\mapsto \tcm(h)\) is a homomorphism to the isotropy group bundle \(H=\bigsqcup_{x\in G^{(0)}} H_x\) of our groupoid~\(G\).  Furthermore, the groupoid~\(G\) acts on the group bundle~\(H\) by a conjugation action:
\[
\xymatrix@C+2em{
  x &
  \ar@/_1pc/[l]_{g}_{}="2"
  \ar@/^1pc/[l]^{g}^{}="3"
  \ar@{=>}"2";"3"^{1_g}
  y &
  \ar@/_1pc/[l]_{1_y}_{}="0"
  \ar@/^1pc/[l]^{\tcm(a)}^{}="1"
  \ar@{=>}"0";"1"^{a}
  y &
  \ar@/_1pc/[l]_{g^{-1}}_{}="2"
  \ar@/^1pc/[l]^{g^{-1}}^{}="3"
  \ar@{=>}"2";"3"^{1_{g^{-1}}}
  x}
\quad\mapsto\quad
\xymatrix@C+2em{
  y & \ar@/_1pc/[l]_{f}_{}="4"
  \ar@/^1pc/[l]^{g\tcm(a)g^{-1}}^{}="5"
  \ar@{=>}"4";"5"^{\acm_g(a)}
  x.
}
\]
The groupoid~\(G\), the group bundle \(H\) with this conjugation action, and the map \(\tcm\colon H\to G\) form a \emph{crossed module} of groupoids.  If we start with a topological \(2\)\nb-groupoid, this yields a crossed module of topological groupoids.  Conversely, any (topological) strict \(2\)\nb-groupoid arises in this fashion for a unique crossed module of (topological) groupoids.  Crossed modules and their (strict) actions on \(\Cst\)\nb-algebras are also studied in~\cite{Buss-Meyer-Zhu:Non-Hausdorff_symmetries}.  In this article, we focus on weak notions of group action and equivariant map.

\begin{example}
  \label{exa:torus_by_integers}
  Let \(\vartheta\in\R\) and let \(\lambda\defeq \Eul^{2\pi\ima\vartheta}\). Consider the map~\(\tcm\) from ~\(\Z\) into the \(1\)\nb-torus~\(\Torus\) defined by \(\tcm(n)\defeq \lambda^n\).  The map \(\tcm\colon \Z\to\Torus\) and the trivial conjugation action of~\(\Torus\) on~\(\Z\) define a crossed module, which appears as symmetry of the rotation algebra \(A_\vartheta\) in~\cite{Buss-Meyer-Zhu:Non-Hausdorff_symmetries}.

  From the discussion above, we can construct a strict \(2\)\nb-group out of this crossed module.  It has a single object~\(\star\), and the group of arrows is~\(\Torus\) with its usual multiplication.  The set of bigons is \(\Z\times\Torus\), where a pair \((n,z)\in \Z\times\Torus\) is viewed as a \(2\)\nb-morphism from~\(z\) to~\(\lambda^n\cdot z\):
  \[
  \xymatrix@C+2em{
    \star & \star
    \ar@/_1pc/[l]_{z}_{}="0"
    \ar@/^1pc/[l]^{\lambda^n\cdot z}_{}="1"
    \ar@{=>}"0";"1"^{n}
  }
  \]
  The horizontal multiplication of \(2\)\nb-morphisms is
  \[
  \xymatrix@C+2em{
    \star & \star
    \ar@/_1pc/[l]_-{z_1}_{}="0"
    \ar@/^1pc/[l]^-{\lambda^{n_1}z_1}_{}="1"
    \ar@{=>}"0";"1"^{n_1}
    &\star
    \ar@/_1pc/[l]_-{z_2}_{}="2"
    \ar@/^1pc/[l]^-{\lambda^{n_2}z_2}_{}="3"
    \ar@{=>}"2";"3"^{n_2}
  }
  \quad\mapsto\quad
  \xymatrix@C+2em{
    \star &\star,
    \ar@/_1pc/[l]_-{z_1z_2}^{}="0"
    \ar@/^1pc/[l]^-{\lambda^{n_1+n_2}z_1z_2}_{}="1"
    \ar@{=>}"0";"1"|{\vphantom{|}n_1+n_2}
  }
  \]
  and the vertical composition is
  \[
  \xymatrix@C+2em{
    \star & \star
    \ar@/_1.5pc/[l]_-{z}^{}="0"
    \ar[l]|{\lambda^n z}^{}="3"_{}="1"
    \ar@/^1.5pc/[l]^-{\lambda^{n'}\lambda^nz}_{}="2"
    \ar@{=>} "0";"1"^{n}
    \ar@{=>} "3";"2"^{n'}
  }
  \quad\mapsto\quad
  \xymatrix@C+2em{
    \star & \star.
    \ar@/_1.5pc/[l]_-{z}^{}="0"
    \ar@/^1.5pc/[l]^-{\lambda^{n+n'}z}_{}="1"
    \ar@{=>}"0";"1"|{\vphantom{|}n+n'}
  }
  \]
\end{example}

\subsection{The weak \texorpdfstring{$2$}{2}-category of correspondences}
\label{sec:Corrcat}

Why did we choose unitary multipliers as our bigons in~\(\Csttwocat\)?  This choice becomes more natural after introducing another \(2\)\nb-category~\(\Corrcat\) of \(\Cst\)\nb-algebras that has \(\Cst\)\nb-correspondences as arrows.

\begin{definition}
  \label{def:correspondence}
  A \emph{\(\Cst\)\nb-correspondence} from~\(A\) to~\(B\) is a Hilbert \(B\)\nb-module~\(\Hilm\) with a non-degenerate \Star{}representation of~\(A\) on~\(\Hilm\) by adjointable operators.
\end{definition}

Bigons in~\(\Corrcat\) are isomorphisms of \(\Cst\)\nb-correspondences, that is, \(A,B\)\nb-bimodule isomorphisms that are unitary for the \(B\)\nb-valued inner products.

\begin{example}
  \label{exa:automorphism_Hilbert_bimodule}
  Let \(f\colon A\to\Mult(B)\) be a \Star{}representation.  This yields a correspondence~\([f]\) that involves~\(B\) with its usual right Hilbert \(B\)\nb-module structure and the left action of~\(A\) induced by~\(f\), that is, \(a\cdot b\defeq f(a)b\); here we tacitly use the fact that the algebra of adjointable operators on~\(B\) is \(\Mult(B)\).  An isomorphism between \([f_1]\) and~\([f_2]\) is a unitary adjointable operator~\(b\) on~\(B\) --~that is, a unitary multiplier of~\(B\)~-- such that \(b f_1(a) = f_2(a) b\) for all \(a\in A\).  Thus the bigons \([f_1]\Rightarrow [f_2]\) in~\(\Corrcat\) are the same as the bigons \(f_1\Rightarrow f_2\) in~\(\Csttwocat\).
\end{example}

To get a \(2\)\nb-category~\(\Corrcat\), we also have to compose arrows and bigons.  The composition of arrows is the (balanced) tensor product: given two composable correspondences \(f_1\colon B\to \Bound(\Hilm_D)\) and \(f_2\colon A\to \Bound(\Hilm_B)\), we want to define
\begin{equation}
  \label{eq:compose_correspondences}
  f_1\circ f_2 \defeq f_2\otimes_B f_1.
\end{equation}
More precisely, the right hand side denotes the Hilbert \(D\)\nb-module \(\Hilm_B \otimes_B \Hilm_D\) with the action \(f_2\otimes_B\Id\) of~\(A\).  Recall that \(\Hilm_B \otimes_B \Hilm_D\) is the completion of the algebraic tensor product \(\Hilm_B\otimes\Hilm_D\) with respect to the \(D\)\nb-valued inner product
\[
\langle \xi_1\otimes\eta_1,\xi_2\otimes\eta_2\rangle
\defeq \bigl\langle \eta_1,
f_1\bigl(\langle \xi_1,f_1\xi_2\rangle_B\bigr)\cdot\eta_2 \bigr\rangle_D.
\]
If \([f_1]\) and~\([f_2]\) are correspondences coming from \Star{}representations \(f_1\colon B\to \Mult(D)\) and \(f_2\colon A\to \Mult(B)\) as in Example~\ref{exa:automorphism_Hilbert_bimodule}, then there is a natural isomorphism
\[
[f_1\circ f_2] \cong [f_1]\circ[f_2] \defeq [f_2]\otimes_B[f_1].
\]
The horizontal and vertical compositions of arrows are given by
\[
b_1\horizprod b_2 \defeq b_2 \otimes_B b_1
\qquad\text{and}\qquad
b_1\vertprod b_2 \defeq b_1 \cdot b_2.
\]
The unit arrow~\(1_A\) on an object~\(A\) (a \(\Cst\)\nb-algebra) is the identity correspondence \([\Id_A]\), and the unit bigon~\(1_f\) on an arrow~\(f\) (a \(\Cst\)\nb-correspondence) is the identity isomorphism on the underlying Hilbert module of~\(f\).  These obvious definitions produce only a \emph{weak} \(2\)\nb-category, also called a \emph{bicategory} in~\cite{Benabou:Bicategories}.

\subsubsection{The definition of weak \texorpdfstring{$2$}{2}-categories}
\label{sec:weak-2-cat}

Weak \(2\)\nb-categories differ from strict ones because the composition of arrows is only associative and unital \emph{up to isomorphism}.  A weak \(2\)\nb-category has objects, arrows and bigons as above, with a unit arrow on each object and a unit bigon on each arrow, a composition of arrows, and horizontal and vertical compositions of bigons.  But the associativity law \((f_1\cdot f_2)\cdot f_3 = f_1 \cdot (f_2\cdot f_3)\) and the unit law \(1\cdot f = f = f\cdot 1\) are \emph{weakened}: we only require invertible bigons
\begin{equation}
  \label{eq:weak_two-category_bigons}
  (f_1\cdot f_2) \cdot f_3
  \xRightarrow[\cong]{a(f_1,f_2,f_3)}
  f_1\cdot (f_2 \cdot f_3),\qquad
  1_y\cdot f \xRightarrow[\cong]{l_f} f
  \xLeftarrow[\cong]{r_f} f\cdot 1_x
\end{equation}
for any triple of composable arrows \(f_1,f_2,f_3\) and any arrow \(f\colon x\to y\).  The bigons \(a(f_1,f_2,f_3)\), \(l_f\) and~\(r_f\) are sometimes called \emph{associator}, and left and right \emph{unitor}, respectively.  Besides naturality, we require the following two diagrams to commute:
\begin{equation}
  \label{eq:pentagon}
  \begin{gathered}
    \xymatrix{
      \bigl((f_1\cdot f_2)\cdot f_3\bigr)\cdot f_4 \ar@{=>}[r] \ar@{=>}[d]&
      (f_1\cdot f_2)\cdot (f_3\cdot f_4) \ar@{=>}[r]&
      f_1\cdot \bigl(f_2\cdot (f_3\cdot f_4)\bigr)\\
      \bigl(f_1\cdot (f_2\cdot f_3)\bigr)\cdot f_4 \ar@{=>}[rr]&&
      f_1\cdot \bigl((f_2\cdot f_3)\cdot f_4\bigr) \ar@{=>}[u]
    }
  \end{gathered}
\end{equation}
for a quadruple of composable arrows \(f_1,f_2,f_3,f_4\), and
\begin{equation}
  \label{eq:coherence-on-1}
  \begin{gathered}
    \xymatrix{
      (f\cdot 1_y)\cdot g \ar@{=>}[rr] \ar@{=>}[dr]&&
      f\cdot (1_y\cdot g) \ar@{=>}[dl]\\
      &f\cdot g
    }
  \end{gathered}
\end{equation}
for composable arrows \(f\colon y\to z\) and \(g\colon x\to y\); the bigons involved come from~\eqref{eq:weak_two-category_bigons}.  By Mac Lane's Coherence Theorem~\cite{MacLane:Associativity}, the two \emph{coherence laws} \eqref{eq:pentagon} and~\eqref{eq:coherence-on-1} imply that any two parallel bigons constructed out of the bigons in~\eqref{eq:weak_two-category_bigons} are equal.

In addition to the coherence laws, the vertical composition of bigons is required to be strictly associative and unital; the horizontal composition of bigons must be associative up to the vertical products with the associators for arrows that are needed to identify the source and range arrows of \((b_1\horizprod b_2)\horizprod b_3\) with those of \(b_1\horizprod (b_2\horizprod b_3)\), and it must be unital in a similar sense.

The following definition summarises these properties:

\begin{definition}
  \label{def:weak_2-cat}
  A \emph{weak \(2\)\nb-category}~\(\cC\) contains the following data:
  \begin{itemize}
  \item a set of objects~\(\cC_0\);
  \item for each pair of objects \((x,y)\), a category \(\cC(x,y)\) of arrows from~\(x\) to~\(y\); the objects and morphisms of the categories~\(\cC(x,y)\) are called arrows and bigons of~\(\cC\), respectively; in particular, this gives us a unit bigon \(\Id_f\) for each arrow~\(f\) in~\(\cC\) and a vertical composition
    
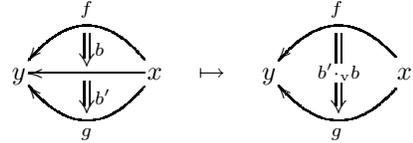
\[
    \xymatrix@1@C+2em{
      y & x
      \ar@/_1.5pc/[l]_-{f}^{}="0"
      \ar[l]_{}="1"
      \ar@/^1.5pc/[l]^-{g}_{}="2"
      \ar@{=>} "0";"1"^{b}
      \ar@{=>} "1";"2"^{b'}
    }
    \quad\mapsto\quad
    \xymatrix@1@C+2em{
      y & x.
      \ar@/_1.5pc/[l]_-{f}^{}="0"
      \ar@/^1.5pc/[l]^-{g}_{}="1"
      \ar@{=>}"0";"1"|{\vphantom{|}b'\vertprod b}
    }
    \]
  \item For each triple \((x, y, z)\) of objects, a composition functor \(\cC(y, z) \times \cC(x, y) \to \cC(x, z)\),
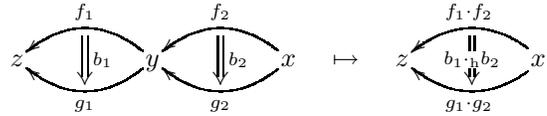
    \[
    \xymatrix@1@C+2em{
      z& y
      \ar@/_1pc/[l]_-{f_1}_{}="0"
      \ar@/^1pc/[l]^-{g_1}_{}="1"
      \ar@{=>}"0";"1"^{b_1}
      &x
      \ar@/_1pc/[l]_-{f_2}_{}="2"
      \ar@/^1pc/[l]^-{g_2}_{}="3"
      \ar@{=>}"2";"3"^{b_2}
    }
    \quad\mapsto\quad
    \xymatrix@1@C+2em{
      z & x,
      \ar@/_1pc/[l]_-{f_1\cdot f_2}^{}="0"
      \ar@/^1pc/[l]^-{g_1\cdot g_2}_{}="1"
      \ar@{=>}"0";"1"|{\vphantom{|}b_1\horizprod b_2}
    }
    \]
    which combines the composition of arrows and the horizontal composition of bigons.

  \item For each object~\(x\), a unit arrow \(1_x \in \cC(x, x)\);
  \item invertible bigons as in~\eqref{eq:weak_two-category_bigons}.
  \end{itemize}
  This data must satisfy the coherence conditions \eqref{eq:pentagon} and~\eqref{eq:coherence-on-1}.
\end{definition}

\begin{example}
  In the case of~\(\Corrcat\), we use the obvious isomorphisms of correspondences
  \[
  (f_1\cdot f_2) \cdot f_3
  \xRightarrow[\cong]{a(f_1,f_2,f_3)}
  f_1\cdot (f_2 \cdot f_3),\qquad
  f\cdot 1_B \xRightarrow[\cong]{l_f} f
  \xLeftarrow[\cong]{r_f} 1_A\cdot f;
  \]
  the first isomorphism maps the generator \((\xi_1\otimes \xi_2)\otimes \xi_3\) of the first tensor product to \(\xi_1\otimes (\xi_2\otimes \xi_3)\); the second and third isomorphisms map \(\xi\otimes b\) and \(a\otimes \xi\) to \(\xi\cdot b\) and \(a\cdot \xi\), respectively.  The coherence laws are trivial to verify for these natural isomorphisms.
\end{example}

The definition of a weak \(2\)\nb-category is a special case of weakening equalities of arrows in \(2\)\nb-categories to isomorphisms of arrows.  Whenever we do this, we must specify the bigons that implement these isomorphisms as part of our data, and we must require these bigons to satisfy suitable coherence laws.  It is usually easy to find \emph{some} such coherence laws.  But it may be more difficult to justify that we have found \emph{all} relevant coherence laws.  The examples we are going to consider are sufficiently well-known to find this in the literature.

\subsection{Isomorphisms and equivalences}

\begin{definition}
  \label{def:isomorphism_equivalence}
  An \emph{isomorphism} in a strict \(2\)\nb-category is an arrow \(f\colon x\to y\) for which there is an arrow \(g\colon y\to x\) with \(g\circ f = 1_x\) and \(f\circ g = 1_y\).

  An \emph{equivalence} in a weak \(2\)\nb-category is an arrow \(f\colon x\to y\) for which there is an arrow \(g\colon y\to x\) with \(g\circ f \cong 1_x\) and \(f\circ g \cong 1_y\), that is, there are invertible bigons \(g\circ f\Rightarrow 1_x\) and \(f\circ g\Rightarrow 1_y\).
\end{definition}

Clearly, products of equivalences are equivalences, and products of isomorphisms are isomorphisms.  The distinction between isomorphisms and equivalences is uninteresting in weak \(2\)\nb-categories because unit arrows in such categories only behave nicely up to equivalence, anyway.

\begin{proposition}
  \label{pro:equivalence_Cstar}
  Let \(A\) and~\(B\) be \(\Cst\)\nb-algebras.  A strictly continuous unital \Star{}homo\-morphism \(f\colon \Mult(A)\to\Mult(B)\) is an equivalence in~\(\Csttwocat\) if and only if it is an isomorphism in~\(\Csttwocat\), if and only if~\(f\) restricts to a \(\Cst\)\nb-algebra isomorphism between~\(A\) and~\(B\).
\end{proposition}

\begin{proof}
  It is clear that the strictly continuous extension of an isomorphism \(A\congto B\) is an isomorphism and hence an equivalence in~\(\Csttwocat\).  Conversely, let~\(f\) be an equivalence in the \(2\)\nb-category~\(\Csttwocat\).  This means that there is another strictly continuous unital \Star{}homomorphism \(g\colon \Mult(B)\to\Mult(A)\) such that \(f\circ g\cong\Id_{\Mult(B)}\) and \(g\circ f\cong\Id_{\Mult(A)}\), that is, \(f\circ g=\Ad_u\) and \(g\circ f=\Ad_v\) for unitaries \(u\in\Mult(B)\), \(v\in\Mult(A)\).  Strict continuity implies that \(g(B)\cdot A\) is norm-dense in~\(A\).  Hence \(f(A)\) is contained in the norm-closure of
  \[
  f\bigl(g(B)\cdot A\bigr) = \Ad_v(B)\cdot f(A) = B\cdot f(A)
  \subseteq B.
  \]
  Thus \(f(A)\subseteq B\) and, similarly, \(g(B)\subseteq A\).  Since inner automorphisms restrict to bijections on \(A\) and~\(B\), this implies that~\(f\) restricts to a bijection \(A\congto B\), so that~\(f\) is an isomorphism of \(\Cst\)\nb-algebras as asserted.
\end{proof}

\begin{proposition}
  \label{pro:equivalence_correspondence}
  Let \(A\) and~\(B\) be two \(\Cst\)\nb-algebras.  A correspondence \(f\colon A\to\Bound(\Hilm)\) from~\(A\) to~\(B\) is an equivalence in~\(\Corrcat\) if and only if~\(f\) is an isomorphism onto \(\Comp(\Hilm)\), so that we may enrich~\(f\) to an \(A,B\)-imprimitivity bimodule.
\end{proposition}

\begin{proof}
  This is equivalent to \cite{Echterhoff-Kaliszewski-Quigg-Raeburn:Categorical}*{Lemma 2.4}.
\end{proof}

\section{The group case}
\label{sec:group_case}

Let~\(G\) be a group.  A (strict) \emph{group action} of~\(G\) on an object~\(A\) of a category~\(\Cat\) is given by invertible arrows~\(\alpha_g\) for all \(g\in G\) that satisfy the equations \(\alpha_1=1_A\) and \(\alpha_{g_1}\alpha_{g_2} = \alpha_{g_1g_2}\) for all \(g_1,g_2\in G\); the equation \(\alpha_1=1_A\) is redundant if~\(\Cat\) is a groupoid.  A (strictly) \emph{equivariant map} between group actions \((\alpha_g)_{g\in G}\) on~\(A\) and \((\beta_g)_{g\in G}\) on~\(B\) is an arrow \(f\colon A\to B\) that satisfies the equation \(f\alpha_g = \beta_g f\) for all \(g\in G\).  Two such maps \(f_1\) and~\(f_2\) are \emph{equal} if \(f_1=f_2\).

In this section, we replace~\(\Cat\) by a \(2\)\nb-category~\(\Cattwo\) and define the notions of \emph{weak group actions}, \emph{weakly equivariant maps}, and \emph{modifications} of weakly equivariant maps by replacing equalities \(f=g\) of arrows by bigons \(f\Rightarrow g\) or \(g\Rightarrow f\) that satisfy appropriate coherence laws.  We limit our discussion to the more palatable group case in this section.  In the next section, we will generalise our definitions, allowing general \(2\)\nb-groupoids to act.  This also explains the coherence laws: they are part of the standard way to define the \(3\)\nb-category of \(2\)\nb-categories (see~\cite{Leinster:Basic_Bicategories}); weak group actions are the same thing as morphisms from a group, viewed as a \(2\)\nb-category, to another \(2\)\nb-category; weakly equivariant maps are transformations between such morphisms, and modifications appear under the same name in~\cite{Leinster:Basic_Bicategories}.

We show that weak group actions in the \(2\)\nb-category~\(\Csttwocat\) are Busby--Smith twisted actions (as defined in~\cite{Kaliszewski:Morita_twisted}), while weak group actions in the correspondence \(2\)\nb-category~\(\Corrcat\) are equivalent to saturated Fell bundles.  Thus \(2\)\nb-categories make precise in what sense Fell bundles are a kind of group action.  It is not clear, however, whether non-saturated Fell bundles can be interpreted as group actions as well: they seem closer to the actions of the inverse semigroup \(S(G)\) defined by Ruy Exel in~\cite{Exel:PartialActionsGroupsAndInverseSemigroups}, whose actions are partial actions of~\(G\), than to actions of the group itself.

A strictly equivariant map between two (weak) group actions is a strict transformation between them in the notation of~\cite{Leinster:Basic_Bicategories}.  Weakening this yields transformations between morphisms.  Any transformation to \(\Csttwocat\) may be decomposed into a strictly equivariant map and an equivalence between two actions on the same object, where equivalence of weak actions in~\(\Csttwocat\) is the usual notion of (exterior) equivalence for Busby--Smith twisted actions.  In addition, covariant representations are a special case of transformations between weak actions.

In the setting of \(2\)\nb-categories, transformations become themselves objects of a category; the morphisms between them are called \emph{modifications}.  For instance, a modification between two covariant representations, viewed as strong transformations, is a unitary intertwiner between these covariant representations.

\subsection{Weak group actions}
\label{sec:weak_group_action}

Let~\(G\) be a discrete group and let~\(\Cattwo\) be a \(2\)\nb-category.  A \emph{weak action} of~\(G\) on an object~\(A\) of~\(\Cattwo\) consists of
\begin{itemize}
\item arrows \(\alpha_g\colon A\to A\) for all \(g\in G\),

\item a bigon \(u\colon  1_A \Rightarrow \alpha_1 \), and

\item bigons \(\omega(g_1,g_2)\colon \alpha_{g_1}\alpha_{g_2}
  \Rightarrow\alpha_{g_1g_2}\) for all \(g_1,g_2\in G\),
\end{itemize}
subject to certain coherence laws.  These bigons replace the equations \(\alpha_1 = \Id_A\) and \(\alpha_{g_1}\alpha_{g_2}=\alpha_{g_1g_2}\) for strict actions by isomorphisms.  Roughly speaking, these isomorphisms satisfy a coherence law whenever we can prove an equality for strict group actions in two different ways.  For instance, we may simplify~\(\alpha_1\alpha_g\) to~\(\alpha_g\) in two different ways: \(\alpha_1\alpha_g = 1_A\alpha_g = \alpha_g\) or \(\alpha_1\alpha_g = \alpha_{1\cdot g} = \alpha_g\).  Similarly, there are two ways to simplify~\(\alpha_g\alpha_1\) to~\(\alpha_g\).  This leads to the coherence laws
\begin{equation}
  \label{eq:group_coherence_unit}
  \begin{gathered}
    \xymatrix{
      \alpha_1\cdot\alpha_g \ar@{=>}[r]^-{\omega(1,g)}
      \ar@{<=}[d]_{u\horizprod\alpha_g}&
      \alpha_{1\cdot g} \ar@{=}[d]\\
      1_A\cdot\alpha_g \ar@{<=>}[r]&
      \alpha_g
    }\qquad
    \xymatrix{
      \alpha_g\cdot\alpha_1 \ar@{=>}[r]^-{\omega(g,1)}
      \ar@{<=}[d]_{\alpha_g\horizprod u}&
      \alpha_{g\cdot 1} \ar@{=}[d]\\
      \alpha_g\cdot1_A \ar@{<=>}[r]&
      \alpha_g.
    }
  \end{gathered}
\end{equation}
This diagram contains some conventions that we will use in all following diagrams to keep them more readable.  First, the unlabeled invertible bigons~\(\Leftrightarrow\) denote the bigons~\eqref{eq:weak_two-category_bigons} that implement the associativity or the unitality of the target category~\(\Cattwo\); they become equalities in a strict \(2\)\nb-category such as~\(\Csttwocat\) and are obvious canonical isomorphisms in \(\Corrcat\).  Secondly, the composition of bigons is the vertical multiplication.  Thirdly, we denote the identity bigon on an arrow~\(\alpha\) by~\(\alpha\) as well.  This explains the notation \(u\horizprod\alpha_g\) and \(\alpha_g\horizprod u\).  The horizontal products that we need are usually of this form: we multiply horizontally with some identity bigon to change the source and target arrows of a bigon.

Similarly, we may simplify
\((\alpha_{g_1}\alpha_{g_2})\alpha_{g_3}\)
to~\(\alpha_{g_1g_2g_3}\) in two ways:
\begin{align*}
  (\alpha_{g_1}\alpha_{g_2})\alpha_{g_3}
  &= \alpha_{g_1g_2}\alpha_{g_3}
  = \alpha_{g_1g_2g_3},\\
  (\alpha_{g_1}\alpha_{g_2})\alpha_{g_3}
  &= \alpha_{g_1}(\alpha_{g_2}\alpha_{g_3})
  = \alpha_{g_1}\alpha_{g_2g_3}
  = \alpha_{g_1g_2g_3}.
\end{align*}
This leads us to the coherence law
\begin{equation}
  \label{eq:group_coherence_associativity}
  \begin{gathered}
    \xymatrix@C-2.5em@R-.5em{
      (\alpha_{g_1}\alpha_{g_2})\alpha_{g_3} \ar@{<=>}[rr]
      \ar@{=>}[d]_{\omega(g_1,g_2)\horizprod \alpha_{g_3}}&&
      \alpha_{g_1}(\alpha_{g_2}\alpha_{g_3})
      \ar@{=>}[d]^{\alpha_{g_1}\horizprod\omega(g_2,g_3)}\\
      \alpha_{g_1g_2}\alpha_{g_3} \ar@{=>}[dr]_{\omega(g_1g_2,g_3)}
      &&
      \alpha_{g_1}\alpha_{g_2g_3} \ar@{=>}[dl]^{\omega(g_1,g_2g_3)}
      \\
      &\alpha_{g_1g_2g_3}
    }
  \end{gathered}
\end{equation}
In the above definition, it is reasonable to require the bigons \(u\) and~\(\omega\) to be invertible; since all bigons in the categories \(\Csttwocat\) and~\(\Corrcat\) are invertible, anyway, we assume this wherever we need it.

It turns out that we do not need more coherence laws because, in some sense, they generate all other reasonable coherence laws.  More precisely, let \(g_1,\dotsc,g_n\in G\).  We use~\(u\) to add as many unit morphisms to this list as we like and then insert brackets to interpret the product \(\alpha_{g_1}\dotsm \alpha_{g_n}\); the associativity bigons in~\eqref{eq:weak_two-category_bigons} tell us how to relate different ways of putting these brackets.  Then we use the bigons~\(\omega\), together with the invertible bigons in~\eqref{eq:weak_two-category_bigons}, to simplify \(\alpha_{g_1}\dotsm \alpha_{g_n}\) to \(\alpha_{g_1\dotsm g_n}\).  The coherence laws above ensure that we get the same bigon \(\alpha_{g_1}\dotsm \alpha_{g_n} \Rightarrow \alpha_{g_1\dotsm g_n}\), no matter how we put brackets to begin with and in which order we simplify our product.  Our definition of weak action is a special case of the more general concept of \emph{morphisms} between weak \(2\)\nb-categories defined in \cite{Benabou:Bicategories}*{Section~4}.

Assume that the bigons \(u\) and~\(\omega\) are invertible.  Then the arrows~\(\alpha_g\) must be \emph{equivalences} because of the \emph{invertible} bigons \(\alpha_g\alpha_{g^{-1}} \Rightarrow \alpha_{gg^{-1}} = \alpha_1 \Leftarrow 1_A\).  In general, the arrows~\(\alpha_g\) need not be invertible.

We call a group action \emph{strict} if the bigons \(u\) and~\(\omega\) are identity bigons.  This yields group actions of~\(G\) in the usual sense.

\begin{remark}
  Recall that the condition \(\alpha_1=\Id\) is redundant if the target category~\(\Cat\) of a group action~\(\alpha\) is a groupoid.  Similarly, if the target \(2\)\nb-category~\(\Cattwo\) is a \(2\)\nb-groupoid, that is, all arrows are equivalences and all bigons are invertible, then the bigon~\(u\) in our definition of a weak action is redundant.  The (invertible) bigon \(\omega(1,1)\colon \alpha_1\alpha_1\Rightarrow\alpha_1\) then yields a canonical invertible bigon \(\alpha_1\Leftarrow \Id\) by cancelling the equivalence~\(\alpha_1\).  More precisely, we follow the chain of bigons
  \[
  \alpha_1 \Leftrightarrow \alpha_1(\alpha_1\alpha_1^{-1})
  \Leftrightarrow (\alpha_1\alpha_1)\alpha_1^{-1}
  \xLeftrightarrow{\omega(1,1)} \alpha_1\alpha_1^{-1} \Leftrightarrow \Id,
  \]
  where the outer two bigons are part of the structure of the inverse~\(\alpha_1^{-1}\), see Definition~\ref{def:isomorphism_equivalence}.  It can be checked that the coherence laws imply that~\(u\) must be this particular bigon.  Hence we may omit the unitary~\(u\) and the coherence law~\eqref{eq:group_coherence_unit} if~\(\Cattwo\) is a \(2\)\nb-groupoid.  That is, a weak action of~\(G\) on an object of~\(\Cattwo\) is equivalent to the pair \((\alpha,\omega)\) satisfying~\eqref{eq:group_coherence_associativity}.  We will repeat this argument below for the simpler case \(\Cattwo=\Csttwocat\).
\end{remark}

\subsubsection{Weak actions of groups by \texorpdfstring{\Star{}}{*-}representations and Busby--Smith twisted actions}
\label{sec:weak_action_Cst}

Now we study weak group actions in the strict \(2\)\nb-category~\(\Csttwocat\) of \(\Cst\)\nb-algebras with \Star{}representations as arrows and unitary intertwiners as bigons.  Here any bigon is invertible.  Furthermore, the isomorphisms in~\(\Csttwocat\) are the \Star{}isomorphisms.  Thus a strict group action on an object of~\(\Csttwocat\) is a group homomorphism from~\(G\) to the group of \Star{}automorphisms of a \(\Cst\)\nb-algebra~\(A\) --~this is the usual notion of a group action on a \(\Cst\)\nb-algebra.

What is a weak group action \((A,\alpha,\omega,u)\) on an object of~\(\Csttwocat\)?  Here~\(A\) is a \(\Cst\)\nb-algebra and \(\alpha_g\colon \Mult(A)\to\Mult(A)\) restricts to a \Star{}automorphism of~\(A\) by Proposition~\ref{pro:equivalence_Cstar} because~\(\alpha_g\) is an equivalence for any weak action.  The bigons \(\omega(g_1,g_2)\) and~\(u\) are unitary multipliers of~\(A\) that satisfy
\begin{alignat}{2}
  \label{eq:weak_action_omega}
  \omega(g_1,g_2) \cdot \alpha_{g_1}\bigl(\alpha_{g_2}(a)\bigr) \cdot \omega(g_1,g_2)^*
  &= \alpha_{g_1g_2}(a)
  &\qquad&\text{for all \(g_1,g_2\in G\), \(a\in A\),}\\
  \label{eq:weak_action_u}
  u\cdot a \cdot u^* &= \alpha_1(a)
  &\qquad&\text{for all \(a\in A\)}
\end{alignat}
because \(u\) and~\(\omega(g_1,g_2)\) are bigons \(1_A\Rightarrow \alpha_1\) and \(\alpha_{g_1}\alpha_{g_2}\Rightarrow \alpha_{g_1g_2}\), respectively.

The coherence laws \eqref{eq:group_coherence_unit} and~\eqref{eq:group_coherence_associativity} look more familiar if we express them in terms of the adjoints \(\omega^*(g_1,g_2)\defeq \omega(g_1,g_2)^*\).  This is what we will do from now on.  Recall that the vertical multiplication in~\(\Csttwocat\) is just multiplication of unitaries, while~\eqref{eq:horizprod_Cstartwocat} yields \(f \horizprod u = f(u)\) and \(u\horizprod f=u\) for an arrow~\(f\) and a unitary~\(u\), viewed as a bigon between some arrows it intertwines.

As a result, the two coherence laws in~\eqref{eq:group_coherence_unit} mean that
\begin{equation}
  \label{eq:weak_action_unit}
  \omega^*(1,g) = u,\qquad
  \omega^*(g,1) = \alpha_g(u)
\end{equation}
for all \(g\in G\); and the coherence law~\eqref{eq:group_coherence_associativity} amounts to the cocycle condition
\begin{equation}
  \label{eq:Busby_Smith_cocycle}
  \alpha_{g_1}\bigl(\omega^*(g_2,g_3)\bigr) \cdot \omega^*(g_1,g_2g_3)
  = \omega^*(g_1,g_2) \cdot \omega^*(g_1g_2,g_3)
\end{equation}
for all \(g_1,g_2,g_3\in G\).  In particular, \(u = \omega^*(1,1)\).

If we define \(u\defeq \omega^*(1,1)\), then \eqref{eq:weak_action_u} follows from~\eqref{eq:weak_action_omega} and the invertibility of~\(\alpha_1\), and the conditions~\eqref{eq:weak_action_unit} follow from~\eqref{eq:Busby_Smith_cocycle} for \(g_2=g_3=1\) and \(g_1=g_2=1\).  Thus a weak action of a group~\(G\) on~\(A\) is given by \Star{}automorphisms~\(\alpha_g\) for \(g\in G\) and unitaries~\(\omega^*(g_1,g_2)\) for \(g_1,g_2\in G\) that satisfy \eqref{eq:weak_action_omega} and~\eqref{eq:Busby_Smith_cocycle}.  This is called a \emph{Busby--Smith twisted dynamical system} in~\cite{Kaliszewski:Morita_twisted}.  The original definition of twisted actions by Busby and Smith in \cite{Busby-Smith:Representations_twisted_group}*{Definition 2.1} imposes the additional condition \(u=1\), that is, \(\alpha_1=\Id_A\) and \(\omega(g,1)=1=\omega(1,g)\).  Lemma~\ref{lem:weaken_unit} will show that any weak action is weakly isomorphic to one with \(\alpha_1=\Id_A\) and \(u=\omega(1,1)^*=1\) (weak isomorphism in \(\Csttwocat\) is the same as exterior equivalence).  Thus the two definitions of Busby--Smith twisted actions in \cites{Busby-Smith:Representations_twisted_group, Kaliszewski:Morita_twisted} are essentially equivalent.

\subsubsection{Group actions by correspondences and saturated Fell bundles}

Now we replace \(\Csttwocat\) by the correspondence category \(\Corrcat\) and study weak \emph{group actions by correspondences}.  We will show that such group actions are equivalent to saturated Fell bundles.  Since the category \(\Corrcat\) is weak, it is pointless to study strict group actions in \(\Corrcat\).

\begin{definition}
  \label{def:Fell_bundle}
  A \emph{Fell bundle} (see~\cite{Doran-Fell:Representations}) over a (discrete) group~\(G\) is a family of Banach spaces~\(A_g\) for \(g\in G\) with multiplication maps \(\mu(g_1,g_2)\colon A_{g_1}\times A_{g_2} \to A_{g_1g_2}\) and conjugate-linear \Star{}operations \(A_g \to A_{g^{-1}}\) that satisfy analogues of the usual conditions for a \(\Cst\)\nb-algebra: the multiplication is associative, the \Star{}operation is an involutive anti-homomorphism with \(\xi^*\xi\geq 0\) for all \(\xi\in A_g\), and the norm satisfies \(\norm{\xi}^2_{A_g} = \norm{\xi\xi^*}_{A_1} = \norm{\xi^*\xi}_{A_1}\) for all \(g\in G\).  A Fell bundle is called \emph{saturated} if the span of \(A_{g_1}\cdot A_{g_2}\) is dense in~\(A_{g_1g_2}\) for all \(g_1,g_2\in G\).
\end{definition}

\begin{theorem}
  \label{the:group_act_Corrcat}
  A group action by correspondences of a group~\(G\) on a \(\Cst\)\nb-algebra~\(A\) is equivalent to a saturated Fell bundle \((A_g)_{g\in G}\) over~\(G\) with a \(\Cst\)\nb-algebra isomorphism \(A\cong A_1\).
\end{theorem}

\begin{proof}
  Let~\((A_g)_{g\in G}\) be a saturated Fell bundle.  Then~\(A_1\) is a \(\Cst\)\nb-algebra, so that it makes sense to ask for a \(\Cst\)\nb-algebra isomorphism \(\varphi\colon A\congto A_1\).  Each~\(A_g\) is a Hilbert \(A_1,A_1\)\nb-bimodule via the multiplication \(A_1\times A_g\times A_1\to A_g\) and the right and left inner products \(\braket{\xi}{\eta}_\rightsub \defeq \xi^*\cdot\eta\) and \(\braket{\xi}{\eta}_\leftsub \defeq \xi\cdot\eta^*\) for \(\xi,\eta\in A_g\).  This Hilbert bimodule is an imprimitivity bimodule because the Fell bundle is saturated.  We use the isomorphism \(\varphi\colon A\to A_1\) to view~\(A_g\) as an \(A,A\)\nb-imprimitivity bimodule or, equivalently, as an invertible correspondence from~\(A\) to itself (Proposition~\ref{pro:equivalence_correspondence}), and let \(\alpha_g \defeq A_{g^{-1}}\).  The bigon \(u\colon A\Rightarrow \alpha_1\) is simply \(\varphi\colon A\to \alpha_1\) --~this \(\Cst\)\nb-algebra isomorphism is an isomorphism of correspondences as well.

  The associativity of the multiplication \(\mu\colon A_{g_1}\times A_{g_2}\to A_{g_1g_2}\) in the Fell bundle yields
  \[
  \mu(a_1\cdot\xi\cdot a_2,\eta\cdot a_3) =
  a_1\cdot\mu(\xi,a_2\cdot\eta)\cdot a_3
  \qquad\text{for all \(a_1,a_2,a_3\in A\).}
  \]
  Hence \(\mu\) induces an \(A\)\nb-bimodule homomorphism \(A_{g_1}\otimes_A A_{g_2} \to A_{g_1g_2}\) that is isometric with respect to both the left and the right inner products.  This isometry is unitary because the Fell bundle is saturated.  We let
  \[
  \omega(g_1,g_2)\colon \alpha_{g_2}\otimes_A \alpha_{g_1} =
  A_{g_2^{-1}} \otimes_A A_{g_1^{-1}} \to A_{g_2^{-1}g_1^{-1}} =
  \alpha_{g_1g_2}
  \]
  be this isomorphism for \(g_2^{-1}\) and~\(g_1^{-1}\).

  The diagram~\eqref{eq:group_coherence_unit} commutes for trivial reasons, and the associativity of the multiplication in the Fell bundle implies the commutativity of~\eqref{eq:group_coherence_associativity}.  Hence the data \((A,\alpha,\omega,u)\) defines a weak group action of~\(G\) in~\(\Corrcat\).

  Conversely, a group action on~\(A\) by correspondences consists of correspondences \(\alpha_g\colon A\to A\) and unitary bimodule homomorphisms
  \[
  u\colon 1_A \Rightarrow \alpha_1,\qquad
  \omega(g_1,g_2)\colon \alpha_{g_2} \otimes_A \alpha_{g_1}
  \Rightarrow \alpha_{g_1g_2}.
  \]

  The correspondences~\(\alpha_g\) are equivalences and hence \(A\)\nb-imprimitivity bimodules by Proposition~\ref{pro:equivalence_correspondence}.  The Banach spaces \(A_g\defeq \alpha_{g^{-1}}\) will, of course, become the fibres of our Fell bundle.  The bigon \(\varphi=u\) identifies \(A \cong \alpha_1\) as an imprimitivity bimodule and, in particular, as a Banach space.

  The unitary intertwiner \(\omega(g_2^{-1},g_1^{-1})\) yields a bilinear map
  \[
  \mu(g_1,g_2)\colon A_{g_1}\times A_{g_2} \to A_{g_1g_2},\qquad
  (\xi_1,\xi_2)\mapsto \omega(g_2^{-1},g_1^{-1})(\xi_1\otimes\xi_2).
  \]
  We use these maps~\(\mu(g_1,g_2)\) to define the multiplication in our Fell bundle.  The multiplication is associative because of the coherence law~\eqref{eq:group_coherence_associativity}.  Since \(\omega(g_1,g_2)\) is unitary, the range of~\(\mu(g_1,g_2)\) is dense in~\(A_{g_1g_2}\) for all \(g_1,g_2\in G\).

  In particular, the above multiplication turns~\(A_1\) into an algebra with multiplication provided by \(\mu(1, 1)\), which we denote by~\(\mu\) for short.  The coherence law~\eqref{eq:group_coherence_unit} implies \(\mu(\xi,\eta) = u^{-1}(\xi)\cdot\eta\) for all \(\xi, \eta\in A_1\).  Applying~\(u^{-1}\) to this equation and using that it is a bimodule homomorphism, we get \(u^{-1}\bigl(\mu(\xi,\eta)\bigr) = u^{-1}(\xi)\cdot u^{-1}(\eta)\), so that \(\varphi=u\colon A\to A_1\) is an algebra homomorphism for the multiplication~\(\mu\) on~\(\alpha_1\).

  Before we can construct the \Star{}operation of our Fell bundle, we must discuss inversion of imprimitivity bimodules.  To reduce confusion, we do this for two possibly different \(\Cst\)\nb-algebras \(A\) and~\(B\).  Let~\(\Hilm\) be an \(A,B\)\nb-imprimitivity bimodule.  The \emph{dual} \(B,A\)\nb-imprimitivity bimodule~\(\Hilm^*\) is~\(\Hilm\) as a set, with the same \(A\)- and \(B\)\nb-valued inner products, the conjugate-linear \(\C\)\nb-vector space structure, and the \(B,A\)-bimodule structure \(b\cdot \xi^*\cdot a\defeq a^*\cdot\xi\cdot b^*\) for \(a\in A\), \(b\in B\), \(\xi^*\in\Hilm^*\); here we write~\(\xi^*\) instead of~\(\xi\) to emphasise that we view~\(\xi\) as an element of~\(\Hilm^*\).  This dual imprimitivity bimodule is inverse to~\(\Hilm\) in the sense that the \(A\)- and \(B\)\nb-valued inner products on~\(\Hilm\) induce canonical isomorphisms \(\Hilm\otimes_B \Hilm^*\cong A\) and \(\Hilm^*\otimes_A \Hilm\cong B\) of Hilbert bimodules.

  The inverse of an imprimitivity bimodule is determined uniquely up to isomorphism.  We can make this more precise in our \(2\)\nb-categorical setup.  Let~\(\Hilm[K]\) be a \(B,A\)\nb-imprimitivity bimodule and let \(v\colon \Hilm\otimes_B\Hilm[K] \congto A\) be an isomorphism.  Then we get an induced isomorphism
  \[
  \hat{v}\colon \Hilm[K]
  \congto B\otimes_B \Hilm[K]
  \congto (\Hilm^*\otimes_A \Hilm)\otimes_B \Hilm[K]
  \congto \Hilm^*\otimes_A (\Hilm\otimes_B \Hilm[K])
  \xrightarrow[\cong]{v} \Hilm^*\otimes_A A
  \congto \Hilm^*,
  \]
  which maps \(\braket{\xi}{\eta}_B\cdot\zeta \mapsto \xi^*\cdot v(\eta\otimes\zeta) = (v(\eta\otimes\zeta)^*\cdot\xi)^*\) for \(\xi,\eta\in\Hilm\), \(\zeta\in\Hilm[K]\).

  We now apply this construction to the isomorphism
  \[
  v_g \defeq u^{-1}\circ\omega(g^{-1},g)\colon
  \alpha_g \otimes_A \alpha_{g^{-1}}
  \xrightarrow{\omega(g^{-1},g)}
  \alpha_{g^{-1}g} = \alpha_1
  \xrightarrow{u^{-1}}
  A.
  \]
  This yields a canonical isomorphism of \(A,A\)-imprimitivity bimodules \(\hat{v}_g\colon A_g\congto A_{g^{-1}}^*\).  Since \(A_{g^{-1}}^*\) is equal to~\(A_{g^{-1}}\) as a set, we may view~\(\hat{v}_g\) as a map from~\(A_g\) to~\(A_{g^{-1}}\).  This is the involution of our Fell bundle: \(\xi^*\defeq \hat{v}_g(\xi)\).  The map \(\xi\mapsto\xi^*\) is conjugate-linear by construction.

  Next we check \((\xi^*)^*=\xi\) for all \(\xi\in A_g\).  This follows from a more general observation in the case of an \(A,B\)\nb-imprimitivity bimodule~\(\Hilm\) and a \(B,A\)\nb-imprimitivity bimodule~\(\Hilm[K]\) with isomorphisms \(v\colon \Hilm\otimes_B\Hilm[K] \to A\) and \(w\colon \Hilm[K]\otimes_A\Hilm \to B\).  Then we get isomorphisms \(\hat{v}\colon \Hilm[K]\to\Hilm^*\) and \(\hat{w}\colon \Hilm\to\Hilm[K]^*\).  We may also view~\(\hat{w}\) as a map \(\hat{w}\colon \Hilm^*\to\Hilm[K]\).  When are \(\hat{w}\) and~\(\hat{v}\) inverse to each other?

  This is clearly the case if \(\Hilm[K]=\Hilm^*\) and \(v\) and~\(w\) are the canonical isomorphisms given by the inner products on~\(\Hilm\).  We may assume that \(\Hilm[K]=\Hilm^*\) and that~\(v\) is this canonical isomorphism because any triple \((\Hilm[K],v,w)\) is isomorphic to one of this form.  Then~\(\hat{w}\) is the inverse of~\(\hat{v}\) if and only if~\(w\) is equal to the \emph{canonical} isomorphism \(\Hilm^*\otimes_A\Hilm\congto B\) induced by the \(B\)\nb-valued inner product on~\(\Hilm\).  Equivalently, the map \(\Id_{\Hilm}\otimes_B w\colon \Hilm\otimes_B\Hilm[K]\otimes_A\Hilm\to\Hilm\) is equal to \(v\otimes_A\Id_{\Hilm}\).  This final formulation makes sense for general \(\Hilm[K]\) and~\(v\).  As a consequence, \(\hat{v}\) and~\(\hat{w}\) are inverse to each other if and only if
  \[
  v\otimes_A\Id_{\Hilm} = \Id_{\Hilm}\otimes_B w\colon
  \Hilm\otimes_B\Hilm[K]\otimes_A\Hilm \to \Hilm.
  \]
  Furthermore, in this case \(\braket{\hat{v}(\xi)}{\eta}_B = w(\xi\otimes\eta)\) for all \(\xi\in\Hilm[K]\), \(\eta\in\Hilm\) because this property is isomorphism-invariant and clearly holds if \(\Hilm[K]=\Hilm^*\) and \(v\) and~\(w\) are the canonical isomorphisms.

  We apply this in the situation \(\Hilm=\alpha_g\), \(\Hilm[K]=\alpha_{g^{-1}}\), \(v=v_g\), \(w=v_{g^{-1}}\).  The coherence laws for a weak group action imply \(v_g\otimes_A \Id_{\alpha_g} = \Id_{\alpha_g}\otimes_A v_{g^{-1}}\) because the following diagram commutes:
  \[
  \xymatrix@C+3em{
    \alpha_g\otimes_A\alpha_{g^{-1}}\otimes_A\alpha_g
    \ar@{=>}[r]^-{\omega(g^{-1},g)\otimes_A\alpha_g}
    \ar@{=>}[d]_{\alpha_g\otimes_A\omega(g,g^{-1})}&
    \alpha_{g^{-1}g} \otimes_A \alpha_g
    \ar@{=>}@/^1pc/[d]^{u^{-1}\otimes_A \alpha_g}
    \ar@{=>}[d]_{\omega(g,g^{-1}g)}\\
    \alpha_g \otimes_A \alpha_{gg^{-1}}
    \ar@{=>}[r]^{\omega(gg^{-1},g)}
    \ar@{=>}@/_1pc/[r]_{\alpha_g\otimes_A u^{-1}}&
    \alpha_g
  }
  \]
  Hence \(\hat{v}_g\) and~\(\hat{v}_{g^{-1}}\) are inverse to each other, that is, \((\xi^*)^*=\xi\) for all \(\xi\in A_g = \alpha_{g^{-1}}\).  Furthermore, we get \(\norm{\xi}^2 = \norm{\braket{\xi}{\xi}} = \norm{\xi^*\cdot\xi}\) for all \(\xi\in A_g\).

  If \(g,h\in G\), then the isomorphism
  \[
  \alpha_{g^{-1}}\otimes_A \alpha_{h^{-1}} \otimes_A \alpha_{gh}
  \xRightarrow{\omega(h^{-1},g^{-1})\otimes_A \alpha_{gh}}
  \alpha_{(gh)^{-1}} \otimes_A \alpha_{gh}
  \xRightarrow{\omega(gh,(gh)^{-1})}
  \alpha_1 \xRightarrow{u^{-1}} A
  \]
  induces the isomorphism \(A_g\otimes_A A_h \to A_{(gh)^{-1}}^*\), \((\xi\otimes\eta)\mapsto (\xi\cdot\eta)^*\).  The map \((\xi\otimes\eta)\mapsto \eta^*\xi^*\) is associated to the unitary
  \begin{multline*}
    \alpha_{g^{-1}}\otimes_A \alpha_{h^{-1}} \otimes_A \alpha_{gh}
    \xRightarrow{\alpha_{g^{-1}}\otimes_A\alpha_{h^{-1}}\otimes_A \omega(g,h)^{-1}}
    \alpha_{g^{-1}}\otimes_A \alpha_{h^{-1}} \otimes_A \alpha_h \otimes_A \alpha_g
    \\\xRightarrow{\alpha_g^{-1}\otimes_A \omega(h,h^{-1}) \otimes_A \alpha_g}
    \alpha_{g^{-1}} \otimes_A \alpha_1 \otimes_A \alpha_g
    \\\xRightarrow{\alpha_{g^{-1}} \otimes_A u^{-1}\otimes_A \alpha_g}
    \alpha_{g^{-1}} \otimes_A \alpha_g
    \xRightarrow{\omega(g,g^{-1})}
    \alpha_1 \xRightarrow{u^{-1}} A.
  \end{multline*}
  It follows from the coherence laws that both unitaries from \(\alpha_{g^{-1}}\otimes_A \alpha_{h^{-1}} \otimes_A \alpha_{gh}\) to~\(A\) agree.  Hence \((\xi\eta)^* = \eta^*\xi^*\) for all \(g,h\in G\), \(\xi\in A_g\), \(\eta\in A_h\).  In particular, it follows that~\(A_1\) is a \(\Cst\)\nb-algebra.

  Clearly, our constructions of saturated Fell bundles from group actions by correspondences and vice versa are inverse to each other, even up to equality and not only up to isomorphism.
\end{proof}

\begin{remark}
  \label{rem:Fell_from_action}
  A weak group action by \Star{}automorphisms is, in particular, a weak group action by correspondences and hence gives rise to a Fell bundle by Theorem~\ref{the:group_act_Corrcat}.  This yields the familiar construction of Fell bundles from Busby--Smith twisted actions (see~\cite{Exel:TwistedPartialActions}).
\end{remark}

\subsection{Weakly equivariant maps}
\label{sec:weakly_equivariant_maps}

Let~\(\Cattwo\) be a \(2\)\nb-category and let~\(G\) be a discrete group.  Let \((A,\alpha,\omega_A,u_A)\) and \((B,\beta,\omega_B,u_B)\) be two weak actions of~\(G\) on objects \(A\) and \(B\) of~\(\Cattwo\).  A strictly equivariant map is a map \(f\colon A\to B\) with equalities \(\beta_g f = f\alpha_g\) for all \(g\in G\).  The map~\(f\) and these equalities, turned into isomorphisms, provide the data of a \emph{transformation}:
\begin{itemize}
\item an arrow \(f\colon A\to B\) and

\item bigons \(V_g\colon \beta_g f \Rightarrow f\alpha_g\) for all \(g\in G\).
\end{itemize}
In addition, we impose the following two coherence laws.  First,
\begin{equation}
  \label{eq:weakly_equivariant_twist1}
  \begin{gathered}
    \xymatrix{
      \beta_1\cdot f \ar@{=>}[d]^{V_1} \ar@{<=}[r]^{u_B\horizprod f}&
      1_B \cdot f \ar@{<=>}[r]&
      f\\
      f\cdot\alpha_1 \ar@{<=}[r]_{f\horizprod u_A}&
      f\cdot 1_A \ar@{<=>}[ru]&
    }
  \end{gathered}
\end{equation}
corresponds to the two simplifications \(\beta_1\cdot f = 1_B \cdot f = f\) and \(\beta_1\cdot f = f\cdot \alpha_1 = f \cdot 1_A = f\) for a strictly equivariant map~\(f\).  Secondly, the two ways of simplifying \(\beta_{g_1}\beta_{g_2} f\) to~\(f\alpha_{g_1g_2}\) provide the following coherence law:
\begin{equation}
  \label{eq:weakly_equivariant_twist2}
  \begin{gathered}
    \xymatrix@C+2.5em{
      (\beta_{g_1}\beta_{g_2})f \ar@{<=>}[d]
      \ar@{=>}[r]^-{\omega_B(g_1,g_2)\horizprod f}&
      \beta_{g_1g_2}f \ar@{=>}[r]^{V_{g_1g_2}}&
      f\alpha_{g_1g_2}\\
      \beta_{g_1}(\beta_{g_2}f)
      \ar@{=>}[d]_{\beta_{g_1}\horizprod V_{g_2}}&&
      f(\alpha_{g_1}\alpha_{g_2})
      \ar@{=>}[u]_{f\horizprod\omega_A(g_1,g_2)}\\
      \beta_{g_1}(f\alpha_{g_2})\ar@{<=>}[r]&
      (\beta_{g_1}f)\alpha_{g_2}
      \ar@{=>}[r]_{V_{g_1}\horizprod\alpha_{g_2}}&
      (f\alpha_{g_1})\alpha_{g_2} \ar@{<=>}[u]
    }
  \end{gathered}
\end{equation}
Recall that the unlabeled isomorphisms correspond to the canonical invertible bigons~\eqref{eq:weak_two-category_bigons}, which are identities for strict \(2\)\nb-categories.  The definition above is equivalent to the definition of a transformation between morphisms of weak categories in~\cite{Leinster:Basic_Bicategories}.  Thus we need no more coherence conditions than the two above.  Indeed, the two coherence conditions above already imply the following: given any word \((g_1,\dotsc,g_n)\) in~\(G\) and a subset~\(I\) with \(g_i=1\) for \(i\in I\), all bigons \(\beta_{g_1}\dotsb\beta_{g_n} f\Rightarrow f\alpha_{g_1\dotsm g_n}\) constructed out of the natural bigons are equal.

The transformation is called \emph{strong} if the bigons~\(V_g\) are invertible, and \emph{strict} if the bigons~\(V_g\) are identities.  The strict transformations are simply arrows \(A\to B\) that intertwine the additional structure \((\omega_A,u_A)\) and \((\omega_B,u_B)\) and deserve to be called \emph{equivariant maps} between weak group actions.  General transformations will be also called \emph{weakly equivariant maps} in this work. The distinction between transformations and strong transformations does not concern us because all bigons in \(\Csttwocat\) and \(\Corrcat\) are invertible, anyway.

\begin{definition}
  \label{def:equivalence}
  Given a weak action \((B,\beta,\omega,u)\) and a family of \emph{invertible} bigons \(V_g\colon \beta_g\Rightarrow\beta'_g\) for some family of invertible arrows \(\beta_g'\colon B\to B\), we let
  \[
  u' \defeq V_1 u\colon 1_B \xRightarrow{u} \beta_1
  \xRightarrow{V_1} \beta'_1
  \]
  and
  \[
  \omega'(g_1,g_2)\defeq \Bigl( \beta'_{g_1}\cdot\beta'_{g_2}
  \xRightarrow{V_{g_1}^{-1}\horizprod V_{g_2}^{-1}}
  \beta_{g_1}\cdot\beta_{g_2} \xRightarrow{\omega(g_1,g_2)}
  \beta_{g_1g_2} \xRightarrow{V_{g_1g_2}} \beta'_{g_1g_2} \Bigr).
  \]
  It is routine to check the coherence laws \eqref{eq:group_coherence_unit} and~\eqref{eq:group_coherence_associativity}, so that \((B,\beta',\omega',u')\) is another weak action of~\(G\) on~\(B\).  By construction, the identity map~\(\Id_B\) with the bigons~\((V_g)\) is a strong transformation from \((B,\beta',\omega',u')\) to \((B,\beta,\omega,u)\).

  In this event, we call~\(V\) an \emph{equivalence} from \((\beta',\omega',u')\) to \((\beta,\omega,u)\).
\end{definition}

Let us specialise this to \(\Cattwo=\Csttwocat\).  Here the \emph{strict transformations} (equivariant maps) are simply non-degenerate \Star{}homomorphisms \(f\colon A\to\Mult(B)\) with
\[
f\circ\alpha_g = \beta_g\circ f,\qquad f\bigl(\omega_A(g_1,g_2)\bigr) = \omega_B(g_1,g_2)\qquad\text{and}\qquad f(u_A)=u_B.
\]
Since all bigons are invertible, there is no difference between transformations and strong transformations.  A transformation from~\(A\) to~\(B\) consists of a non-degenerate \Star{}homomorphism \(f\colon A\to\Mult(B)\) and unitaries \(V_g\in\Mult(B)\) with
\begin{equation}
  \label{eq:transformation_intertwine}
  V_g \beta_g\bigl(f(a)\bigr)V_g^* = f\bigl(\alpha_g(a)\bigr)
  \qquad\text{for all \(g\in G\)}
\end{equation}
because~\(V_g\) is a bigon from~\(\beta_gf\) to~\(f\alpha_g\).

When we plug in the definitions of horizontal and vertical products in~\(\Csttwocat\), the coherence laws \eqref{eq:weakly_equivariant_twist1} and~\eqref{eq:weakly_equivariant_twist2} amount to the requirements \(V_1 \cdot u_B = f(u_A) \) and
\begin{equation}
  \label{eq:transformation_cocycle}
  f\bigl(\omega_A(g_1,g_2)\bigr) \cdot V_{g_1} \cdot
  \beta_{g_1}(V_{g_2}) = V_{g_1g_2} \cdot \omega_B(g_1,g_2)
  \qquad\text{for all \(g_1,g_2\in G\),}
\end{equation}
where~\(\cdot\) is the multiplication of unitary elements in~\(\Mult(B)\).  (The unlabelled invertible bigons in~\eqref{eq:weakly_equivariant_twist2} are identities in~\(\Csttwocat\).)

Recall that the unitaries \(u_A\) and~\(u_B\) are redundant because \(u_A=\omega_A(1,1)^*\) and \(u_B=\omega_B(1,1)^*\) and that the adjoints \(\omega_A^*\) and~\(\omega_B^*\) yield Busby--Smith twisted actions.  The condition~\eqref{eq:transformation_cocycle} for \(g_1=g_2=1\) specialises to \(f(u_A^*)\cdot V_1 u_BV_1u_B^* = V_1u_B^*\).  Hence the coherence law \(V_1\cdot u_B = f(u_A) \) is redundant.  Thus \(f\) and \((V_g)_{g\in G}\) form a transformation from \((A,\alpha,\omega_A,u_A)\) to \((B,\beta,\omega_B,u_B)\) if and only if they satisfy~\eqref{eq:transformation_intertwine} and the cocycle condition
\[
V_{g_1} \cdot \beta_{g_1}(V_{g_2}) \cdot \omega_B^*(g_1,g_2) =
f\bigl(\omega^*_A(g_1,g_2)\bigr) \cdot V_{g_1g_2}
\qquad\text{for all \(g_1,g_2\in G\).}
\]

A \emph{strictly equivariant map} is a transformation \((f,V)\) with \(V_g=1\) for all \(g\in G\).  The conditions above become \(f\bigl(\alpha_g(a)\bigr) = \beta_g\bigl(f(a)\bigr)\) for all \(a\in A\), \(g\in G\), and \(\omega^*_B(g_1,g_2) = f\bigl(\omega^*_A(g_1,g_2)\bigr)\), that is, \(f\) intertwines the group actions and preserves the twists.

Let \((\beta',\omega')\) and \((\beta,\omega)\) be weak actions of~\(G\) on the same \(\Cst\)\nb-algebra~\(B\).  Recall that an equivalence~\(V\) between them is a transformation \((f,V)\) with \(f=\Id_B\).  This means that we are given unitaries \(V_g\in\Mult(B)\) for all \(g\in G\) that satisfy
\begin{alignat}{2}
  \label{eq:outer_equivalence_intertwine}
  \beta_g'(b) &= V_g \beta_g(b) V_g^*
  &\quad&\text{for all \(b\in B\), \(g\in G\),}\\
  \label{eq:outer_equivalence_cocycle}
  \omega^{\prime *}(g_1,g_2) &=
  V_{g_1} \cdot \beta_{g_1}(V_{g_2}) \cdot
  \omega^*(g_1,g_2) \cdot V_{g_1g_2}^*
  &\quad&\text{for all \(b\in B\), \(g_1,g_2\in G\).}
\end{alignat}
Two weak actions are called equivalent if there is an equivalence between them.  This is the notion of exterior equivalence for Busby-Smith twisted actions in \cite{Busby-Smith:Representations_twisted_group}*{Definition~2.4}.

The bigons~\(V_g\) are, in particular, unitary multipliers of~\(B\) and provide bigons \(\beta_g\Rightarrow\beta_g'\) for some~\(\beta_g'\).  Hence the construction in Definition~\ref{def:equivalence} yields:

\begin{lemma}
  \label{lem:decompose_strong_strict_equivalence}
  If \(\Cattwo\) is~\(\Csttwocat\), then a transformation \((f,V)\) from \((A,\alpha,\omega_A,u_A)\) to \((B,\beta,\omega_B,u_B)\) decomposes into a strict transformation \((f,1)\) from \((A,\alpha,\omega_A,u_A)\) to \((B,\beta',\omega'_B,u'_B)\) and an equivalence~\(V\) from \((B,\beta',\omega'_B,u'_B)\) to \((B,\beta,\omega_B,u_B)\).
\end{lemma}

As a result, transformations between Busby--Smith twisted actions combine strictly equivariant maps with equivalence of twisted actions.

\begin{lemma}
  \label{lem:weaken_unit}
  If~\(\Cattwo\) is~\(\Csttwocat\), then any weak action \((\beta,\omega,u)\) on a \(\Cst\)\nb-algebra~\(B\) is equivalent to a weak action \((\beta',\omega',u')\) with \(\beta'_1=\Id_B\) and \(u'=1\).
\end{lemma}

\begin{proof}
  Let \(V_g\defeq u^{-1}\) for all \(g\in G\).  These unitary multipliers provide invertible bigons \(V_g\colon \beta_g\Rightarrow\beta_g'\) for some arrows~\(\beta_g'\colon B\to B\).  Definition~\ref{def:equivalence} provides an equivalence to a weak action \((\beta',\omega',u')\) with \(u'\defeq V_1 u=1\) and hence \(\beta'_1=\Id_B\).
\end{proof}

\begin{example}
  \label{exa:covariant_representation_as_weakly_equivariant_map}
  Now we specialise to the case where~\(B\) carries the trivial action, \(\beta_g=\Id_B\) for all \(g\in G\) and \(\omega_B^*(g_1,g_2)=1\) for all \(g_1,g_2\in G\).  In this case, the conditions for a transformation are
  \[
  f\bigl(\alpha_g(a)\bigr) = V_g f(a)V_g^*,\qquad
  V_{g_1} \cdot V_{g_2} = f\bigl(\omega^*_A(g_1,g_2)\bigr) \cdot V_{g_1g_2}.
  \]
  Such a pair \((f,V)\) is called a \emph{covariant representation} of a weak group action, compare \cite{Busby-Smith:Representations_twisted_group}*{Definition after Theorem 3.2} for Busby--Smith twisted actions.
\end{example}

Next, we describe transformations between weak actions of~\(G\) by correspondences, that is, we specialise the concepts above to the \(2\)\nb-category~\(\Corrcat\).  By Theorem~\ref{the:group_act_Corrcat}, weak actions of~\(G\) by correspondences are equivalent to saturated Fell bundles over~\(G\).  We are going to relate transformations to Morita equivalences of saturated Fell bundles.

\begin{definition}
  \label{def:Morita_equivalence_Fell_bundles}
  A \emph{Morita equivalence} between two saturated Fell bundles \(A=(A_g)_{g\in G}\) and \(B=(B_g)_{g\in G}\) over a (discrete) group~\(G\) is a Banach bundle \(\Gamma = (\Gamma_g)_{g\in G}\) over~\(G\) such that
  \begin{itemize}
  \item there is a non-degenerate \(G\)\nb-grading preserving \(A,B\)-bimodule structure on~\(\Gamma\) in the sense that there are bilinear maps \(A_{g_1}\times\Gamma_{g_2}\to \Gamma_{g_1g_2}\) and \(\Gamma_{g_1}\times B_{g_2}\to \Gamma_{g_1g_2}\), written multiplicatively, such that
    \begin{itemize}
    \item \(a_1\cdot (a_2\cdot\xi)=(a_1a_2)\cdot\xi\) for all \(a_1,a_2\in A\) and \(\xi\in \Gamma\);
    \item \((\xi\cdot b_1)\cdot b_2=\xi\cdot(b_1b_2)\) for all \(b_1,b_2\in B\) and \(\xi\in \Gamma\);
    \item \((a\cdot\xi)\cdot b=a\cdot(\xi\cdot b)\) for all \(a\in A\), \(\xi\in \Gamma\), and \(b\in B\); and
    \item the multiplication map \(A_{g_1}\otimes\Gamma_{g_2}\otimes B_{g_3}\to \Gamma_{g_1g_2g_3}\) has dense range for all \(g_1,g_2,g_3\in G\).
    \end{itemize}
  \item there are bilinear maps (left and right \emph{inner products})
    \[
    \braket{\cdot}{\cdot}_A\colon
    \Gamma_{g_1}\times\Gamma_{g_2}^*\to
    A_{g_1g_2^{-1}}
    \quad\text{and}\quad
    \braket{\cdot}{\cdot}_B\colon
    \Gamma_{g_1}^*\times\Gamma_{g_2}\to B_{g_1^{-1}g_2}
    \]
    for all \(g_1,g_2\in G\), such that
    \begin{itemize}
    \item \(\braket{\xi\cdot b}{\eta}_A = \braket{\xi}{\eta\cdot b^*}_A\) and \(\braket{a\cdot\xi}{\eta}_B=\braket{\xi}{a^*\cdot\eta}_B\) for all \(\xi,\eta\in\Gamma\), \(a\in A\), \(b\in B\);
    \item \(\braket{\xi}{\eta}_B^*=\braket{\eta}{\xi}_B\) and \(\braket{\xi}{\eta}_A^*=\braket{\eta}{\xi}_A\) for all \(\xi,\eta\in \Gamma\);
    \item \(\braket{\xi}{\xi}_B\ge0\), \(\braket{\xi}{\xi}_A\ge0\), and \(\norm{\xi}^2 = \norm{\braket{\xi}{\xi}_A} = \norm{\braket{\xi}{\xi}_B}\) for all \(\xi\in\Gamma\);
    \item \(\braket{\xi}{\eta\cdot b}_B = \braket{\xi}{\eta}_B\cdot b\) and \(\braket{a\cdot\xi}{\eta}_A = a\cdot \braket{\xi}{\eta}_A\) for all \(\xi,\eta\in \Gamma\), \(a\in A\) and \(b\in B\); and
    \item \(\braket{\xi}{\eta}_A\cdot\zeta = \xi\cdot \braket{\eta}{\zeta}_B\) for all \(\xi,\eta,\zeta\in \Gamma\);
    \item \(\braket{\Gamma_1}{\Gamma_1}_A\) is dense in~\(A_1\) and \(\braket{\Gamma_1}{\Gamma_1}_B\) is dense in~\(B_1\).
    \end{itemize}
  \end{itemize}
  A \emph{correspondence} of Fell bundles from~\(A\) to~\(B\) is defined similarly, dropping the \(A\)\nb-valued inner products, all conditions that involve it, and the fullness of the \(B_1\)\nb-valued inner product.
\end{definition}

\begin{example}
  \label{exa:identity_bimodule_Fell}
  Let \(A=\Gamma=B\) for a saturated Fell bundle~\(B\), equip~\(\Gamma\) with the given bimodule structure and the inner products \(\braket{\xi}{\eta}_A \defeq \xi\eta^*\), \(\braket{\xi}{\eta}_B \defeq \xi^*\eta\).  This is the identity Morita equivalence on~\(B\).
\end{example}

Morita equivalence of saturated Fell bundles was introduced by Yamagami in~\cite{Yamagami:IdealStructure} and later used by Muhly and Williams in \cites{Muhly:BundlesGroupoids, Muhly-Williams:Equivalence.FellBundles}.  These articles work with Fell bundles over groupoids, where the notion of Morita equivalence becomes more interesting because it also relates Fell bundles over possibly different groupoids.  Nevertheless, the underlying groupoids are always Morita equivalent by definition.  In the group case, this does not appear because Morita equivalence in this case reduces to isomorphism so that there is no loss of generality in assuming that the two groups are the same as we did above.

\begin{proposition}
  \label{pro:Morita_Fell_transformation}
  Under the equivalence between weak actions of~\(G\) by correspondences and saturated Fell bundles over~\(G\) described in Theorem~\textup{\ref{the:group_act_Corrcat}}, a transformation between two weak actions of~\(G\) is equivalent to a correspondence between the associated saturated Fell bundles.  Equivalences of weak actions correspond to Morita equivalences between the associated Fell bundles.
\end{proposition}

\begin{proof}
  Let \((A,\alpha,\omega_A)\) and \((B,\beta,\omega_B)\) be weak actions of~\(G\).  Recall that the associated saturated Fell bundles have the fibres \(A_g\defeq \alpha_{g^{-1}}\) and \(B_g\defeq \beta_{g^{-1}}\) at~\(g\), respectively.  A transformation from \((A,\alpha,\omega_A)\) to \((B,\beta,\omega_B)\) involves a correspondence~\(\gamma\) from~\(A\) to~\(B\) together with isomorphisms \(V_g\colon \gamma\otimes_B\beta_g \congto \alpha_g\otimes_A\gamma\).  We let
  \[
  \Gamma_g \defeq \alpha_{g^{-1}} \otimes_A \gamma;
  \]
  this is another correspondence from~\(A\) to~\(B\).  Using the isomorphisms \(A\cong\alpha_1 = A_1\) and \(B\cong\beta_1 = B_1\), we view~\(\Gamma_g\) as a correspondence from~\(A_1\) to~\(B_1\).  Using the unitaries \(V\), \(\omega_A\), and~\(\omega_B\), we may construct unitary operators
  \begin{align*}
    A_{g_1} \otimes_A \Gamma_{g_2} &\defeq
    \alpha_{g_1^{-1}}  \otimes_B \alpha_{g_2^{-1}} \otimes_A \gamma
    \xrightarrow{\omega_A}  \alpha_{g_2^{-1}g_1^{-1}} \otimes_B \gamma
    \cong \Gamma_{g_1g_2},\\
    \Gamma_{g_1} \otimes_B B_{g_2} &\defeq
    \begin{multlined}[t]
      \alpha_{g_1^{-1}} \otimes_A \gamma \otimes_B \beta_{g_2^{-1}}
      \xrightarrow{V} \alpha_{g_1^{-1}}\otimes_A
      \alpha_{g_2^{-1}}\otimes_A \gamma
      \\\xrightarrow{\omega_A}  \alpha_{g_2^{-1}g_1^{-1}}\otimes_A\gamma
      \cong \Gamma_{g_1g_2}.
    \end{multlined}
  \end{align*}
  The left \(A_1\)\nb-module structure so defined is evidently equal to the given one.  The same assertion for the right \(B_1\)\nb-module structures is equivalent to the coherence law~\eqref{eq:weakly_equivariant_twist1}.  The associativity of~\(\omega_A\) implies that the maps \(A\otimes \Gamma\to \Gamma\) form a left module structure.

  The associativity of~\(\omega_B\) and the interchange law imply that the left \(A\)- and right \(B\)\nb-module structures on~\(\Gamma\) commute, that is, the following diagram commutes
  \[
  \xymatrix{
    A_{g_1} A_{g_2}\Gamma B_{g_3} \ar@{=>}[r]^{\omega_A}
    \ar@{=>}[d]_{V}&
    A_{g_1g_2}\Gamma B_{g_3} \ar@{=>}[d]^{V}\\
    A_{g_1} A_{g_2} A_{g_3}\gamma
    \ar@{=>}[r]^{\omega_A} \ar@{=>}[d]_{\omega_A}&
    A_{g_1g_2} A_{g_3}\Gamma \ar@{=>}[d]^{\omega_A}\\
    A_{g_1}A_{g_2g_3}\Gamma \ar@{=>}[r]_{\omega_A}&
    A_{g_1g_2g_3}\Gamma.
  }
  \]
  The composition in the top square is the horizontal product of \(V_{g_3}\) and~\(\omega_A(g_1^{-1},g_2^{-1})\).  Here we reversed the order of products, that is, \(A_{g_1}A_{g_2}\) abbreviates \(A_{g_1}\otimes_{A_1} A_{g_2}\), and so on.  We will continue to do this throughout the proof.

  The coherence law~\eqref{eq:weakly_equivariant_twist2} holds if and only if the maps \(\Gamma\otimes B\to\Gamma\) form a right module structure.  This involves the following commuting diagram:
  \[
  \xymatrix{
    A_{g_1}\Gamma B_{g_2} B_{g_3} \ar@{=>}[r]^{V}
    \ar@{=>}[dd]_{\omega_B}&
    A_{g_1} A_{g_2}\Gamma B_{g_3} \ar@{=>}[r]^{\omega_A}
    \ar@{=>}[d]_V&
    A_{g_1g_2}\Gamma B_{g_3} \ar@{=>}[d]^V\\
    &
    A_{g_1} A_{g_2}A_{g_3}\Gamma \ar@{=>}[r]^{\omega_A}
    \ar@{=>}[d]^{\omega_A}&
    A_{g_1g_2}A_{g_3}\Gamma\ar@{=>}[d]^{\omega_A}\\
    A_{g_1}\Gamma B_{g_2g_3} \ar@{=>}[r]_{V}&
    A_{g_1} A_{g_2g_3}\Gamma \ar@{=>}[r]_{\omega_A}&
    A_{g_1g_2g_3}\Gamma
  }
  \]
  The argument above shows that the two small squares on the right commute, and the commutativity of the pentagon on the left is equivalent to the coherence law~\eqref{eq:weakly_equivariant_twist2}.

  To write down the \(B\)\nb-valued inner product, we identify \(\Gamma_g \cong \Gamma\otimes_B B_g\) via~\(V\) and define
  \[
  \Gamma_{g_1}^*\times\Gamma_{g_2} \to B_{g_1^{-1}g_2},\qquad
  \braket{\xi_1\otimes_B b_1}{\xi_2\otimes_B b_2}_B
  \defeq b_1^* \cdot \braket{\xi_1}{\xi_2}_B \cdot b_2.
  \]
  It is straightforward to check the identities \(\braket{\xi_1}{\xi_2}_B^* = \braket{\xi_2}{\xi_1}_B\), \(\braket{\xi_1}{\xi_2\cdot b}_B = \braket{\xi_1}{\xi_2}_B\cdot b\), and \(\braket{a\cdot\xi_1}{\xi_2}_B = \braket{\xi_1}{a^*\cdot \xi_2}_B\) for all \(\xi_1\in\Gamma_{g_1}\), \(\xi_2\in\Gamma_{g_2}\), \(a\in A_{g_3}\), \(B\in B_{g_4}\), \(g_1,g_2,g_3,g_4\in G\).  Furthermore, \(\braket{\xi}{\xi}_B\ge0\) for all~\(\xi\).  The norm on \(\Gamma\otimes_B B_g\) is defined so that \(\norm{\xi}^2 = \norm{\braket{\xi}{\xi}_B}\).  The inner product on~\(\Gamma_1\) coincides with the given \(B\)\nb-valued inner product on~\(\gamma\).  As a result, \(\Gamma=(\Gamma_g)\) with the extra structure defined above is a correspondence of Fell bundles.

  If we start with an equivalence of weak actions, then~\(\gamma\) is even an imprimitivity \(A,B\)-bimodule.  This allows us to define \(A\)\nb-valued inner products as well, satisfying the same conditions as the \(B\)\nb-valued inner products.  The condition \(\braket{\xi}{\eta}_A\zeta = \xi\braket{\eta}{\zeta}_B\) is also built in.

  Conversely, let us start with a correspondence of Fell bundles~\(\Gamma = (\Gamma_g)\) from~\(A\) to~\(B\).  We put \(\gamma\defeq \Gamma_1\); this is a correspondence from~\(A_1\) to~\(B_1\).  Since the multiplication map \(\Gamma_1\times B_g\to \Gamma_g\) has dense range, we get a unitary isomorphism \(\Gamma_g \cong \gamma \otimes_B B_g\).  The multiplication map
  \[
  A_g \times \Gamma_1 \to \Gamma_g \cong \gamma \otimes_B B_g
  \]
  induces an operator \(A_g \otimes_A \Gamma \to \Gamma \otimes_B B_g\).  This operator has dense range because~\((A_g)\) is saturated and \(A_1\times\Gamma \to \Gamma\) has dense range, and it is unitary as well.  Hence we have the data required for a transformation between two weak actions.  The coherence law~\eqref{eq:weakly_equivariant_twist1} is trivial, and the coherence law~\eqref{eq:weakly_equivariant_twist2} follows from the associativity of the bimodule structures on~\(\Gamma\).  Hence we have associated a transformation between weak actions to a correspondence between saturated Fell bundles.  If we start with a Morita equivalence of Fell bundles, then the transformation~\(\gamma\) is invertible, so that we get an equivalence of weak actions.

  It is also clear that the map from transformations to correspondences and back is the identity map on transformations.  The same holds for the map from Fell bundle correspondences to transformations and back because our assumptions for Fell bundle correspondences imply that \(\Gamma_g \cong \Gamma_1 \otimes_B B_g\) and that the inner products \(\Gamma_{g_1}^*\times \Gamma_{g_2}\to B_{g_1^{-1}}B_{g_2}\) are completely determined by the \(B_1\)\nb-inner product on~\(\Gamma_1\).  Analogous statements hold for equivalences of weak actions and Morita equivalences of Fell bundles.
\end{proof}

Recall that transformations from Busby--Smith twisted actions to trivial actions are equivalent to covariant representations (see Example~\ref{exa:covariant_representation_as_weakly_equivariant_map}).  This suggests to interpret transformations from weak actions of~\(G\) by correspondences to trivial actions as covariant representations as well.  Interpreting weak actions by correspondences as saturated Fell bundles by Theorem~\ref{the:group_act_Corrcat}, this leads naturally to covariant representations of Fell bundles.  We interpret these using Proposition~\ref{pro:Morita_Fell_transformation}.  Thus let \(A=(A_g)_{g\in G}\) be a Fell bundle over a group~\(G\) and let~\(B\) be a \(\Cst\)\nb-algebra.  We equip~\(B\) with the trivial action of~\(G\) and interpret the result as a constant Fell bundle with fibre \(B_g=B\) for all \(g\in G\) and the multiplication and involution from~\(B\).  A covariant representation of the saturated Fell bundle~\(A\) on~\(B\) is given by a correspondence of Fell bundles \(\Gamma=(\Gamma_g)_{g\in G}\).  Recall that \(\Gamma_g \cong \Gamma_1 \otimes_B B_g\). Since \(B_g=B\), this simply means \(\Gamma_g\cong\Gamma_1\) for all \(g\in G\).  Thus we are given a single full Hilbert \(B\)\nb-module~\(\Hilm \defeq \Gamma_1\) with bilinear maps \(A_g\times\Hilm\to\Hilm\).  The conditions for a correspondence in Definition~\ref{def:Morita_equivalence_Fell_bundles} amount to requiring that we get a map from the total space \(\A \defeq \bigsqcup_{g\in G} A_g\) to the \(\Cst\)\nb-algebra of adjointable operators on~\(\Hilm\) that preserves both the multiplication and involution. This is the usual definition of a representation of~\(A\):

\begin{definition}[\cite{Doran-Fell:Representations}]
  A (nondegenerate) \emph{representation} of a (saturated) Fell bundle \(A=(A_g)_{g\in G}\) on a Hilbert \(B\)\nb-module~\(\Hilm\) is a family \((\pi_g)_{g\in G}\) of linear maps \(\pi_g\colon A_g\to \Bound(\Hilm)\) satisfying
  \begin{enumerate}[label=(\roman{*})]
  \item \(\pi_{gh}(ab)=\pi_g(a)\pi_h(b)\) for all \(g,h\in G\), \(a\in A_g\) and \(b\in A_h\),
  \item \(\pi_g(a)^*=\pi_{g^{-1}}(a^*)\)  for all \(g\in G\), \(a\in A_g\), and
  \item \(\pi_1(A_1)\Hilm\) spans a dense subspace of~\(\Hilm\).
  \end{enumerate}
\end{definition}

The non-degeneracy condition~(iii) implies that the linear span of \(\pi_g(A_g)\Hilm\) is dense in~\(\Hilm\) for all \(g\in G\) because \(A_gA_{g^{-1}}=A_1\).

\subsection{Modifications}
\label{sec:equiv_weak_equivariant}

Finally, we weaken equality of equivariant maps to the notion of a modification between two transformations.  Let \((A,\alpha,\omega_A,u_A)\) and \((B,\beta,\omega_B,u_B)\) be two weak actions of~\(G\) on objects of a \(2\)\nb-category~\(\Cattwo\) and let \((f,V)\) and \((f',V')\) be transformations (weakly equivariant maps) between them.  A \emph{modification} \((f,V) \Rightarrow (f',V')\) is a bigon \(W\colon f\Rightarrow f'\) --~which weakens the equality \(f=f'\)~-- that satisfies one coherence law
\begin{equation}
  \label{eq:modification_group_action}
  \begin{gathered}
    \xymatrix{
      \beta_g f \ar@{=>}[r]^{V_g}
      \ar@{=>}[d]_{\beta_g\horizprod W}&
      f\alpha_g \ar@{=>}[d]^{W\horizprod \alpha_g}\\
      \beta_g f' \ar@{=>}[r]^{V'_g}&
      f'\alpha_g,
    }
  \end{gathered}
\end{equation}
which corresponds to the two ways of simplifying~\(\beta_g f\) to~\(f' \alpha_g\) via \(\beta_g f'\) or~\(f \alpha_g\) (see also~\cite{Leinster:Basic_Bicategories}).

An invertible modification \(W\colon (f,V)\Rightarrow (f',V')\) (that is, the bigon~\(W\) is invertible) is also called an \emph{equivalence}, and then \((f,V)\) and~\((f',V')\) are called \emph{equivalent}.  In the \(2\)\nb-categories we are interested in, all bigons are invertible, so that any modification is an equivalence.

For a transformation \((f,V)\) between two morphisms and an invertible bigon \(W\colon f\Rightarrow f'\) for some arrow~\(f'\), we define
\[
V'_g \defeq (W\horizprod\alpha_g) V_g (\beta_g\horizprod W^{-1})\colon
\beta_g f' \xRightarrow{\beta_g\horizprod W^{-1}}
\beta_g f \xRightarrow{V_g}
f \alpha_g \xRightarrow{W\horizprod \alpha_g}
f'\alpha_g.
\]
It is routine to check that \((f',V')\) is again a transformation and that~\(W\) is an equivalence \((f,V)\Rightarrow (f',V')\).

Now we consider the category~\(\Csttwocat\).  Let \((f,V)\) and \((f',V')\) be transformations from \((A,\alpha,\omega_A,u_A)\) to \((B,\beta,\omega_B,u_B)\).  Since any bigon in~\(\Csttwocat\) is invertible, any modification is an equivalence.  The coherence law~\eqref{eq:modification_group_action} amounts to the condition \(V'_g \beta_g(W) = WV_g\) for all \(g\in G\).  Thus a modification \((f,V)\Rightarrow (f',V')\) is a unitary multiplier~\(W\) of~\(B\) that satisfies
\begin{equation}
  \label{eq:modification_Cstar}
  f'(a) = W\cdot f(a)\cdot W^*,\qquad
  V'_g = W\cdot V_g\cdot \beta_g(W^*)
  \qquad\text{for all \(a\in A\), \(g\in G\).}
\end{equation}

\begin{example}
  \label{exa:modification_intertwiner}
  Let~\(B\) carry a trivial action as in Example~\ref{exa:covariant_representation_as_weakly_equivariant_map}, so that transformations to~\(B\) are covariant representations.  Then the conditions above simplify to \(f'(a) = Wf(a)W^*\) and \(V'_g = WV_gW^*\), that is, \(W\) is a \emph{unitary intertwiner} between two covariant representations.
\end{example}

Recall that in the correspondence \(2\)\nb-category~\(\Corrcat\), transformations between weak actions are equivalent to correspondences between saturated Fell bundles.  Obviously, there is a bijection between equivalences of transformations and isomorphisms between the associated correspondences of saturated Fell bundles.

\subsection{Adding topology}
\label{sec:topological_group}

Now let~\(G\) be a locally compact topological group.  Then we should impose continuity conditions for weak actions, transformations, and modifications.

If we use the target category~\(\Csttwocat\), we merely have to topologise the groups \(\Aut(A)\) and \(\U\Mult(A)\) of \Star{}automorphisms and unitary multipliers of a \(\Cst\)\nb-algebra~\(A\).  We equip \(\Aut(A)\) with the strong topology, that is, the topology of pointwise convergence.  We equip \(\Mult(A)\) and the subgroup \(\U\Mult(A)\) with the strict topology, that is, the topology generated by the maps \(\lambda_a\colon \Mult(A)\to A\), \(m\mapsto a\cdot m\), and \(\rho_a\colon \Mult(A)\to A\), \(m\mapsto m\cdot a\).  Thus a net \((u_i)_{i\in I}\) in \(\U\Mult(A)\) converges if and only if the nets \((u_i\cdot a)\) and \((a\cdot u_i)\) are norm-convergent for each \(a\in A\).

\begin{definition}
  \label{def:continuous_action_group}
  A weak group action \((A,\alpha,\omega)\) is called \emph{continuous} if the maps \(\alpha\colon G\to\Aut(A)\) and \(\omega\colon G\times G\to\U\Mult(A)\) are continuous with respect to the strong topology on \(\Aut(A)\) and the strict topology on \(\U\Mult(A)\) described above.  A transformation \((f,V)\) between two continuous group actions is called \emph{continuous} if \(V\colon G\to\U\Mult(A)\) is continuous with respect to the strict topology on \(\U\Mult(A)\).  All modifications between group actions are continuous by definition.
\end{definition}

Continuity for weak actions by correspondences is more interesting.  The continuity of a family of correspondences \((\alpha_g)_{g\in G}\) from~\(A\) to~\(B\) is not a property but an additional datum, namely, a \(\Cont_0(G)\)-linear correspondence~\(\alpha\) from \(\Cont_0(G,A)\) to \(\Cont_0(G,B)\).  In the discrete case, we take here the direct sum \(\bigoplus_{g\in G} \alpha_g\) with the obvious \(\Cont_0(G,A)\)-\(\Cont_0(G,B)\)-bimodule structure.  A \(\Cont_0(G)\)-linear correspondence from \(\Cont_0(G,A)\) to \(\Cont_0(G,B)\) is equivalent to a family \((\alpha_g)_{g\in G}\) of correspondences from~\(A\) to~\(B\) together with a space~\(\alpha\) of continuous sections, which forms a correspondence from \(\Cont_0(G,A)\) to \(\Cont_0(G,B)\).  We may also view~\(\alpha\) as an upper semi-continuous Banach bundle.

Thus a continuous action of~\(G\) on~\(A\) by correspondences involves a \(\Cont_0(G)\)-linear correspondence~\(\alpha\) from \(\Cont_0(G,A)\) to itself.  There is no continuity condition for~\(u\), it is simply an isomorphism \(\alpha_1\cong A\), where~\(\alpha_1\) denotes the fibre of~\(\alpha\) at~\(1\) (here we tacitly interpret~\(\alpha\) as a family of correspondences together with a space of continuous sections).  The multiplication maps \(\omega_{g_1,g_2}\colon \alpha_{g_2}\otimes_A \alpha_{g_1} \to \alpha_{g_1g_2}\) should be continuous in the sense that the product of two continuous sections is again continuous.  Equivalently, the maps~\(\omega_{g_1,g_2}\) piece together to a unitary intertwiner
\[
\omega\colon \pi_2^*\alpha\otimes_{\Cont_0(G,A)} \pi_1^*\alpha
\Rightarrow \mu^*\alpha,
\]
where \(\mu\colon G\times G\to G\) is the multiplication map and \(\pi_1,\pi_2\colon G\times G\rightrightarrows G\) are the coordinate projections.  Thus \(\pi_2^*\alpha\otimes_{\Cont_0(G,A)} \pi_1^*\alpha\) and~\(\mu^*\alpha\) are \(\Cont_0(G\times G)\)-linear correspondences with fibres \(\alpha_{g_2}\otimes_A\alpha_{g_1}\) and \(\alpha_{g_1g_2}\) at \((g_1,g_2)\in G\times G\), respectively.  Thus we arrive at the following definition of a continuous weak action by correspondences:

\begin{definition}
  \label{def:continuous_action_correspondence}
  A continuous action of~\(G\) on~\(A\) by correspondences consists of a \(\Cont_0(G)\)\nb-linear correspondence~\(\alpha\) from~\(\Cont_0(G,A)\) to itself and unitary intertwiners
  \[
  \omega\colon \pi_2^*\alpha\otimes_{\Cont_0(G,A)} \pi_1^*\alpha
  \Rightarrow \mu^*\alpha.
  \]
  and \(u\colon [\Id_A] \Rightarrow \alpha_1\) that satisfy analogues of \eqref{eq:group_coherence_unit} and~\eqref{eq:group_coherence_associativity}.  More precisely, we view~\(\alpha\) as an upper semi-continuous field~\((\alpha_g)\) of Hilbert \(A,A\)\nb-bimodules, \(u\) as a unitary intertwiner from~\(\alpha_1\) to~\([\Id_A]\), and~\(\omega\) as a continuous family of unitary intertwiner~\(\omega(g,h)\) from \(\alpha_h\otimes_A \alpha_g\) to~\(\alpha_{hg}\), and require the latter to satisfy these coherence conditions.
\end{definition}

Continuity for a transformation \((f,V)\) between continuous actions \((A,\alpha,\omega_A,u_A)\) and \((B,\beta,\omega_B,u_B)\) by correspondences is defined, similarly, by requiring the unitary intertwiners~\(V_g\) to piece together to a unitary intertwiner \(V\colon f\otimes_B \beta\to \alpha\otimes_A f\).  Once again, continuity is no restriction for modifications.

The results above all extend to the continuous case with little change.

Continuous actions by correspondences still correspond to saturated Fell bundle. Usually, Fell bundles are required to be continuous Banach bundles, but our definition only requires upper semi-continuity.  However, for Fell bundles over groups there is no difference between continuity and upper semi-continuity.  We would like to thank Ruy Exel for explaining the following result to us.

\begin{lemma}
  \label{lem:upperSemiContFellBundle=Continuous}
  Let~\(G\) be a locally compact group and let \(A=(A_g)_{g\in G}\) be an upper semi-continuous Banach bundle.  Suppose there are a multiplication \(\cdot\colon A\times A\to A\) and an involution \({}^*\colon A\to A\) which are continuous for the topology on~\(A\) and satisfy all the algebraic conditions of a Fell bundle.  Then the norm \(a\mapsto \|a\|\) is continuous and hence~\(A\) is a \emph{continuous} Fell bundle in the usual sense.
\end{lemma}

\begin{proof}
  It is enough to observe that the norm \(a\mapsto \|a\|\) is the composition of the continuous maps
  \[
  [a\in A]
  \mapsto [(a^*,a)\in A\times A]
  \mapsto [a^*a\in A_1]
  \mapsto \|a^*a\|^{\frac{1}{2}}=\|a\|.
  \]
  The last map is the norm on the \(\Cst\)\nb-algebra~\(A_1\), which is obviously continuous.
\end{proof}

\begin{theorem}
  \label{the:group_act_Corrcat_continuous}
  A continuous group action by correspondences of a locally compact group~\(G\) on a \(\Cst\)\nb-algebra~\(A\) is equivalent to a saturated Fell bundle \((\alpha_g)_{g\in G}\) over~\(G\) with a \(\Cst\)\nb-algebra isomorphism \(\alpha_1\cong A\).
\end{theorem}

\begin{proof}
  It is well-known that Hilbert modules over \(\Cont_0(G,A)\) correspond to upper semi-continuous bundles of Hilbert \(A\)\nb-modules, that is, upper semi-continuous Banach space bundles with a continuous right \(A\)\nb-module structure and a continuous \(A\)\nb-valued inner product.  Hence our continuous weak action yields an upper semi-continuous Banach bundle \(\A=(\alpha_{g^{-1}})_{g\in G}\).  We construct a multiplication and an involution on~\(\A\) as in the proof of Theorem~\ref{the:group_act_Corrcat}.  These are continuous because the continuity of these algebraic operations can be expressed in terms of continuous sections.  For instance, the multiplication is continuous if and only if the product \((g,h)\mapsto \xi(g)\eta(h)\) of two continuous sections \(\xi,\eta\) of~\(\A\) is a continuous section of the pull-back of~\(\A\) along the multiplication map \(G\times G\to G\).  A similar statement holds for the involution (see \cite{Doran-Fell:Representations}*{VIII.2.4 and VIII.3.2} for further details).  Although the results in~\cite{Doran-Fell:Representations} are stated only for \emph{continuous} bundles, upper semi-continuity suffices for the proofs. By Lemma~\ref{lem:upperSemiContFellBundle=Continuous}, \(\A\) is a continuous Fell bundle.
\end{proof}

Lemmas \ref{lem:decompose_strong_strict_equivalence} and~\ref{lem:weaken_unit} also have analogues for continuous actions.  Of course, in the construction before Lemma~\ref{lem:decompose_strong_strict_equivalence} we require~\((V_g)\) to be continuous.  Finally, continuous transformations to trivial actions and modifications between them correspond to continuous covariant representations and their intertwiners as in Examples \ref{exa:covariant_representation_as_weakly_equivariant_map} and~\ref{exa:modification_intertwiner}.  For weak actions by correspondences, continuous transformations to trivial actions correspond to continuous representations of Fell bundles.  Recall that a representation \(\pi=(\pi_g)_{g\in G}\) of a Fell bundle \(\A=(A_g)_{g\in G}\) on a Hilbert module~\(\Hilm\) is \emph{continuous} if and only if the map \(g\mapsto \pi_g(a_g)\xi\) is continuous from~\(G\) to~\(\Hilm\) for any continuous section~\((a_g)\) of~\(\A\).

Proposition~\ref{pro:Morita_Fell_transformation} also has an analogue in the topological setting: continuous transformations between two (arbitrary) weak actions of~\(G\) by correspondences (resp. equivalences) correspond to \emph{continuous} correspondences (resp. Morita equivalences) between the associated Fell bundles. Continuity of correspondences and Morita equivalences between Fell bundles is defined by requiring, in addition, that the Banach bundle \(\Gamma=\{\Gamma_g\}_{g\in G}\) be upper semi-continuous and the left and right actions and inner products on~\(\Gamma\) be continuous (see~\cite{Muhly-Williams:Equivalence.FellBundles}).

\section{Actions of higher categories}
\label{sec:generalise_higher}

Now we generalise the definitions of weak action, transformation, and modification to weak \(2\)\nb-groupoids instead of groups.  In particular, for strict \(2\)\nb-groups associated to normal subgroups this combines Busby--Smith and Green twists of group actions.

The basic structure is parallel to the special case of group actions.  But we must allow for more than one object and impose further conditions that take care of the bigons in the weak \(2\)\nb-groupoid that acts.

\subsection{Morphisms}
\label{sec:morphisms_two}

A group action on an object of a category may also be viewed as a functor from the group to the category.  In Section~\ref{sec:weak_group_action}, we weakened the notion of group action.  The resulting notion of weak group action is a special case of weak functors, which we get by weakening the notion of a functor.  These weak functors are just called \emph{morphisms} of \(2\)\nb-categories.

A \emph{strict functor} between two strict \(2\)\nb-categories is a triple of maps on objects, arrows and bigons that preserve units and the three products.  Strict functors between strict \(2\)\nb-groupoids correspond bijectively to morphisms of crossed modules; of course, the latter are pairs of morphisms \(G_1\to G_2\), \(H_1\to H_2\) such that the following two diagrams commute:
\[
\xymatrix{
  H_1\ar[d]\ar[r]^{\tcm_1}&G_1\ar[d]\\
  H_2\ar[r]^{\tcm_2}&G_2
}\qquad
\xymatrix{
  G_1\times H_1\ar[d]\ar[r]^-{\acm_1}&H_1\ar[d]\\
  G_2\times H_2\ar[r]^-{\acm_2}&H_2.
}
\]
The equations regarding arrows that may be weakened are the two characteristic properties of a functor regarding composition of arrows and unit arrows.  Correspondingly, a morphism of weak categories has bigons that replace these two equations as additional data.  These are subject to certain coherence laws.  In detail:

\begin{definition}
  \label{def:morphism}
  A \emph{morphism} between two weak \(2\)\nb-categories \(\cC\) and~\(\cC'\) consists of the following data:
  \begin{itemize}
  \item a map \(F\colon \cC_0 \to \cC'_0\) between the object sets;
  \item a functor \(F(x, y)\colon \cC(x, y) \to \cC'\bigl(F(x), F(y)\bigr)\) for each pair of objects \(x,y\);
  \item a natural bigon \(\omega(f_1, f_2)\colon F(f_1) \cdot F(f_2) \Rightarrow F(f_1 \cdot f_2)\) for each pair \((f_1, f_2)\) of composable arrows; naturality means that the following diagrams commute for all pairs of bigons \(n_1\colon f_1\Rightarrow g_1\), \(n_2\colon f_2\Rightarrow g_2\) between composable arrows:
    \begin{equation}
      \label{eq:naturality-composition}
      \begin{gathered}
        \xymatrix@C+2em{
          F(f_1)  \cdot F(f_2) \ar@{=>}[r]^-{\omega(f_1, f_2)}
          \ar@{=>}[d]_{F(n_1)\cdot F(n_2)} & F(f_1 \cdot f_2)
          \ar@{=>}[d]^{F(n_1\cdot n_2)} \\
          F(g_1)  \cdot F(g_2) \ar@{=>}[r]_-{\omega(g_1, g_2)} &   F(g_1 \cdot g_2);
        }
      \end{gathered}
    \end{equation}
    in more categorical language, these are natural transformations \(\omega_{x, y, z}\)
    \[
    \xymatrix{
      \cC(x, y)  \times \cC(y, z) \ar[r]^-{\cdot} \ar[d]_{F(x,y) \times F(y, z)}
      \ar@{=>}[dr]^{\omega_{x, y, z}}&
      \cC(x, z) \ar[d]^{F(x, z)}\\
      \cC'(F(x), F(y)) \times \cC'(F(y), F(z)) \ar[r]_-{\cdot'}&
      \cC'(F(x), F(z))}
    \]
    for all triples of objects \(x\), \(y\), \(z\) of~\(\cC\);

  \item bigons \(u_x\colon 1_{F(x)} \Rightarrow F(1_x)\) for all objects~\(x\).
  \end{itemize}
  The above data is subject to the following coherence conditions:
  \begin{gather}
    \label{eq:wcat_coherence_associativity}
    \begin{gathered}
      \xymatrix@C+7em@R-.5em{
        (F(g_1)\cdot F(g_2))\cdot F(g_3)
        \ar@{=>}[r]^{a'(F(g_1),F(g_2), F(g_3))}
        \ar@{=>}[d]_{\omega(g_1,g_2)\cdot \Id_{F(g_3)}}&
        F(g_1)\cdot (F(g_2)\cdot F(g_3))
        \ar@{=>}[d]^{\Id_{F(g_1)}\cdot\omega(g_2,g_3)}\\
        F(g_1\cdot g_2)F(g_3) \ar@{=>}[d]_{\omega(g_1\cdot g_2,g_3)}
        &
        F(g_1)F(g_2\cdot g_3)
        \ar@{=>}[d]^{\omega(g_1,g_2\cdot g_3)}
        \\
        F((g_1\cdot g_2) \cdot g_3)
        \ar@{=>}[r]_{F(a(g_1,g_2, g_3))} &
        F(g_1\cdot (g_2\cdot g_3))
      }
    \end{gathered}\\
    \label{eq:wcat_coherence_unit}
    \begin{gathered}
      \xymatrix@C+1em{
        F(1_x) \cdot F(g) \ar@{=>}[r]^-{\omega(1_x,g)}
        \ar@{<=}[d]_{u\cdot \Id_{F(g)}}&
        F(1_x\cdot g )\ar@{=>}[d]^{F(l)}\\
        1_{F(x)}\cdot F(g) \ar@{=>}[r]_{l'}&
        F(g)
      }\qquad
      \xymatrix@C+1em{
        F(g)\cdot F(1_y) \ar@{=>}[r]^-{\omega(g,1_y)}
        \ar@{<=}[d]_{\Id_{F(g)}\cdot u}&
        F(g\cdot 1_y) \ar@{=>}[d]^{F(r)}\\
        F(g)\cdot 1_{F(y)} \ar@{=>}[r]_{r'}&
        F(g).
      }
    \end{gathered}
  \end{gather}

  If, in addition, the bigons \(u_x\) and \(\omega(f,g)\) above are all invertible, we speak of a \emph{homomorphism} of weak \(2\)\nb-categories.  If they are all identities, we get a \emph{strict functor} (this notion is only interesting if both \(2\)\nb-categories involved are strict).
\end{definition}

In the \(2\)\nb-categories of \(\Cst\)\nb-algebras introduced above, all bigons are invertible, so that there is no difference between morphisms and homomorphisms into \(\Csttwocat\) and \(\Corrcat\).  Strict functors from strict \(2\)\nb-groupoids into \(\Csttwocat\) are equivalent to the actions of crossed modules studied in~\cite{Buss-Meyer-Zhu:Non-Hausdorff_symmetries}.

\subsubsection{Morphisms to \texorpdfstring{$\Csttwocat$}{C*(2)}}
\label{sec:morphisms_Csttwounit}

Now we make this notion of a morphism concrete for the target category \(\cC=\Csttwocat\) and a weak category~\(G\).  There is no difference between morphisms and homomorphisms because all arrows in~\(\Csttwocat\) are invertible.  A morphism consists of the following data:
\begin{itemize}
\item for each object \(x\in G_0\), a \(\Cst\)\nb-algebra~\(A_x\);
\item for each arrow \(g\in G_1\), a \Star{}isomorphism \(\alpha_g\colon A_{\source(g)} \to A_{\target(g)}\) (use Proposition~\ref{pro:equivalence_Cstar});
\item for each object \(x\in G_0\), a unitary multiplier~\(u_x\) with \(u_x \cdot \alpha_{1_x}(a) \cdot u_x^* =a\) for all \(a \in A_x\);
\item for each pair of composable arrows \((g_1, g_2) \in G_1 \times_{\source, G_0, \target} G_1\), a unitary multiplier~\(\omega(g_1, g_2)\) of~\(A_{\target(g_1)}\) with
  \[
  \omega(g_1,g_2) \cdot \alpha_{g_1}\bigl(\alpha_{g_2}(a)\bigr) \cdot \omega(g_1,g_2)^*
  = \alpha_{g_1g_2}(a)
  \qquad \text{for all \(a\in A_{\source(g_2)}\);}
  \]
\item for each bigon \(n\in G_2\), \(n\colon f\Rightarrow g\), a unitary multiplier~\(\twc_n\) of \(A_{\target(f)}= A_{\target(g)}\) with \(\alpha_g(a) = \twc_n \cdot \alpha_f(a) \cdot\twc_n^*\) for all \(a\in A_{\source(g)}\).
\end{itemize}
These satisfy the following naturality conditions:
\begin{itemize}
\item \(\twc_{m\vertprod n}=\twc_m\cdot\twc_n\) for all composable bigons \(m\) and~\(n\); in particular, this implies \(\twc_{\Id_f}=1\) for all \(f\in G_1\) and \(\twc_{n^{-1}}=\twc_n^*\) for all \(n\in G_2\) which have an inverse;
\item given a horizontal composition in~\(G\),
  \[
  \xymatrix@1@C+2em{
    z &
    \ar@/_1pc/[l]_{f_1}_{}="0"
    \ar@/^1pc/[l]^{g_1}^{}="1"
    \ar@{=>}"0";"1"^{a}
    y &
    \ar@/_1pc/[l]_{f_2}_{}="2"
    \ar@/^1pc/[l]^{g_2}^{}="3"
    \ar@{=>}"2";"3"^{b}
    x}
  \quad\mapsto\quad
  \xymatrix@1@C+2em{
    z & \ar@/_1pc/[l]_{f_1f_2}_{}="4"
    \ar@/^1pc/[l]^{g_1g_2}^{}="5"
    \ar@{=>}"4";"5"^{a\horizprod b}
    x,
  }
  \]
  the following two diagrams in~\(\Csttwocat\) have the same composition:
  \[
  \xymatrix@1@C+2em{
    A_z &
    \ar@/_2pc/[l]_{\alpha_{f_1}}^{}="0"
    \ar[l]^{\alpha_{g_1}}_{}="1"
    \ar@{=>}"0";"1"^{\twc_a}
    A_y &
    \ar@/_2pc/[l]_{\alpha_{f_2}}^{}="2"
    \ar[l]^{\alpha_{g_2}}_{}="3"
    \ar@/^3pc/[ll]^{\alpha_{g_1g_2}}^{}="5"
    \ar@{=>}"2";"3"^{\twc_b}
    \ar@{=>}[l];"5"|{\omega(g_1,g_2)}
    A_x}\qquad
  \xymatrix@1@C+3em{
    A_z & \ar@/_2pc/[l]_{\alpha_{f_1}\alpha_{f_2}}^{}="4"
    \ar[l]|{\alpha_{f_1f_2}}_{}^{}="5"
    \ar@/^2pc/[l]^{\alpha_{g_1 g_2}}_{}="6"
    \ar@{=>}"4";"5"|{\omega(f_1,f_2)}
    \ar@{=>}"5";"6"|{\twc_{a\horizprod b}}
    A_x;
  }
  \]
  that is,
  \begin{equation}
    \label{eq:weak_action_naturality_horizontal}
    \omega(g_1, g_2) \cdot (\twc_a \horizprod \twc_b)
    = \omega(g_1, g_2) \cdot \twc_a\cdot \alpha_{f_1}(\twc_b)
    = \twc_{a \horizprod b} \cdot \omega(f_1, f_2).
  \end{equation}
\end{itemize}
Finally, the coherence conditions \eqref{eq:wcat_coherence_associativity} and~\eqref{eq:wcat_coherence_unit} amount to the following conditions:
\begin{itemize}
\item \(\omega(g_1\cdot g_2,g_3) \cdot \omega(g_1,g_2) = \omega(g_1,g_2\cdot g_3)\cdot \alpha_{g_1}\bigl(\omega(g_2,g_3)\bigr)\cdot \alpha\bigl(a(g_1, g_2, g_3)\bigr)\) for three composable arrows \(g_1\), \(g_2\) and~\(g_3\) in~\(G\), where~\(a\) denotes the associator of~\(G\);
\item \(\twc_{l(g)}\cdot \omega(1_x, g) \cdot u_x = 1\) and \(\twc_{r(g)} \cdot \omega(g, 1_y) \cdot \alpha_g(u_y) = 1\), where~\(g\) is an arrow \(x\to y\) in~\(G\) and \(l\) and~\(r\) denote the unitors in~\(G\).
\end{itemize}
These conditions and data only make explicit the definition of a morphism from a \(2\)\nb-category to~\(\Csttwocat\). This defines weak actions of a weak \(2\)\nb-category~\(G\) in the \(2\)\nb-category \(\Csttwocat\).  It is a routine exercise to do the same for morphisms to \(\Corrcat\).  In fact, the latter category is not much more special than a general weak \(2\)\nb-category, so that almost no simplification of the general definition is possible.

We now specialise to morphisms from strict \(2\)\nb-groupoids to~\(\Csttwocat\) and reformulate the above definition in terms of crossed modules.  This yields a notion of weak \(2\)\nb-groupoid action that combines Green twists and Busby--Smith twists and generalises both to groupoids.

Let~\(\cC\) be a strict \(2\)\nb-groupoid and let \((G, H, \acm, \tcm)\) be the corresponding crossed module of groupoids (see also~\cite{Buss-Meyer-Zhu:Non-Hausdorff_symmetries} for the definition of crossed modules of groupoids).  Such a crossed module acts on \(\Csttwocat\) by \Star{}isomorphisms \(\actA_g\colon A_{\source(g)}\to A_{\target(g)}\) for \(g\in G\) and unitary multipliers \(u_h\in \U\Mult(A_x)\) for \(h\in H_x\), such that the~\((\actA_g)\) form an action of the groupoid~\(G\) and the~\((u_h)\) a homomorphism of group bundles \(H_\bullet\to \U\Mult(A_\bullet)\).  It is easy to check that this is equivalent to a strict homomorphism from~\(\cC\) to~\(\Csttwocat\).  We concentrate on the more difficult case of weak actions of~\(\cC\).

Let~\(X\) be the common object space of \(G\) and~\(H\).  Recall that \(\cC_0=X\), \(\cC_1= G\), and \(\cC_2 = G\times_X H\).  Thus a weak action of~\(\cC\) involves \(\Cst\)\nb-algebras~\(A_x\) for all \(x\in X\), \Star{}isomorphisms \(\alpha_g\colon A_{\source(g)}\to A_{\target(g)}\) for \(g\in G\), and unitary multipliers \(u_x\in \U\Mult(A_x)\) for \(x\in X\), \(\omega(g_1,g_2)\in \U\Mult(A_{\target(g_1)})\) for composable arrows \(g_1,g_2\in G\), and \(\twc_{(h,g)}\in \U\Mult(A_{\target(g)})\) for all \((h,g)\in H\times_X G\).

The data \((A_x,\alpha_g,\omega(g_1,g_2),u_x)\) without the~\(\twc_{(h,g)}\) for \((h,g)\in H\times_X G\) defines a weak action of the groupoid \(\cC_1\rightrightarrows \cC_0\) on \(\Csttwocat\).  We have already identified such weak actions with Busby--Smith twisted actions (the extension from groups to groupoids is straightforward).  In particular, we have seen that the unitaries~\(u_x\) are redundant because \(u_x=\omega(1_x,1_x)\).  Observe that the coherence conditions in the definition of a weak action of~\(\cC\) only involve the actions of \(\cC_1\rightrightarrows \cC_0\).

It remains to analyse the naturality conditions for the unitary multipliers~\(\twc_{(h,g)}\) for \((h,g)\in H\times_X G\).  First we show that these unitaries depend on~\(g\) in a specific way.  We always embed \(H\to\cC_2\) by mapping \(h\in H\) to the bigon \((h,1)\colon 1\to \tcm(h)\), and we let \(v_h \defeq v_{(h,1)}\).  Since the bigon \((h,g)\colon g\to \tcm(h)g\) represents \(h\horizprod1_g\), the naturality conditions above imply first \(\twc_{(1,g)}=1\) for all \(g\in G\) and then
\begin{equation}\label{eq:reduce-g}
\twc_{(h,g)} = \omega(\tcm(h),g)\cdot \twc_h\cdot \omega(1,g)^*.
\end{equation}
Thus it suffices to specify~\(\twc_h\) for \(h\in H\).

If \(\omega(g_1,g_2)=1\) for all \(g_1,g_2\in G\), we get a strict action, which is equivalent to an action of a crossed module.  Such an action is characterised by three conditions.  First, \(\alpha_{\tcm(h)}(a) = \twc_h a \twc_h^*\) for all \(x\in X\), \(h\in H_x\), \(a\in A_x\); secondly, \(\twc_{h_1h_2} = \twc_{h_1}\twc_{h_2}\) for all \(x\in X\), \(h_1,h_2\in H_x\); and thirdly, \(\twc_{\acm_g(h)} = \alpha_g(\twc_h)\) for all \(x\in X\), \(g\in G_x^y\), \(h\in H_x\).

The first condition for a crossed module action expresses that~\(\twc_h\) is a bigon from~\(\Id\) to~\(\alpha_{\tcm(h)}\).  For a general weak action, this becomes
\begin{equation}
  \label{eq:weak_action_crossed_module_1}
  \alpha_{\tcm(h)}(a) = \twc_h \alpha_1(a) \twc_h^*\qquad
  \text{for all \(x\in X\), \(h\in H_x\), \(a\in A_x\).}
\end{equation}
Here \(\alpha_1(a) = \omega(1,1)^* a \omega(1,1)\) for all \(a\in A\).  Equation~\eqref{eq:weak_action_crossed_module_1} ensures that~\(\twc_{(h,g)}\) is a bigon from~\(\alpha_g\) to \(\alpha_{\tcm(h)g}\) for all \((h,g)\in H\times_X G\).

The map \((h,g)\mapsto \twc_{(h,g)}\) is compatible with vertical composition if and only if the following diagram in \(\Csttwocat\) commutes for all \((h_1,h_2) \in H\times_X H\):
\begin{equation}
  \label{eq:weak_action_crossed_module_2}
  \begin{gathered}
    \xymatrix@C+3em{
      \alpha_1 \ar@{=>}[r]^-{\twc_{h_2}}
      \ar@{=>}[d]_-{\twc_{h_1h_2}}&
      \alpha_{\tcm(h_2)} \ar@{=>}[r]^-{\omega(1,\tcm(h_2))^*}&
      \alpha_1\alpha_{\tcm(h_2)} \ar@{=>}[d]^-{\twc_{h_1}}\\
      \alpha_{\tcm(h_1h_2)} \ar@{=}[r]&
      \alpha_{\tcm(h_1)\tcm(h_2)}&
      \alpha_{\tcm(h_1)}\alpha_{\tcm(h_2)} \ar@{=>}[l]^-{\omega(\tcm(h_1),\tcm(h_2))}
    }
  \end{gathered}
\end{equation}
This replaces the condition \(\twc_{h_1}\twc_{h_2} = \twc_{h_1h_2}\) for a strict crossed module action. With some effort, one can verify that \eqref{eq:reduce-g} tells us that \eqref{eq:weak_action_crossed_module_2} is equivalent to \(v_n \cdot v_m = v_{n\vertprod m}\).

The third condition \(\alpha_g(\twc_h) = \twc_{\acm_g(h)}\) for a strict crossed module action ensures that the map \((h,g)\mapsto \twc_{(h,g)}\) is compatible with horizontal products.  For general weak actions, this becomes the following commuting diagram:
\begin{equation}
  \label{eq:weak_action_crossed_module_3}
  \begin{gathered}
    \xymatrix@C+2em{
      \alpha_g\alpha_1\alpha_{g^{-1}} \ar@{=>}[r]^-{\omega(g,1)}
      \ar@{=>}[d]^{\alpha_g(\twc_h)}&
      \alpha_g\alpha_{g^{-1}} \ar@{=>}[r]^-{\omega(g,g^{-1})}&
      \alpha_1 \ar@{=>}[d]^{\twc_{\acm_g(h)}}\\
      \alpha_g\alpha_{\tcm(h)}\alpha_{g^{-1}} \ar@{=>}[r]_-{\omega(g,\tcm(h))}&
      \alpha_{g\tcm(h)}\alpha_{g^{-1}} \ar@{=>}[r]_-{\omega(g\tcm(h),g^{-1})}&
      \alpha_{g\tcm(h)g^{-1}}.
    }
  \end{gathered}
\end{equation}
Take \(a=(h_1, f_1)\), \(b=(h_2, f_2)\). Then \eqref{eq:reduce-g} tells us the relation between \(v_{a\horizprod b}\) and \(v_{h_1 \vertprod c_{f_1}(h_2)}\), and \eqref{eq:weak_action_crossed_module_2} tells us the relation between \(v_{h_1 \vertprod c_{f_1}(h_2)}\) and \(v_{h_1} \cdot v_{c_{f_1}(h_2)}\).  However, since~\(\Csttwocat\) is a strict \(2\)\nb-category, the similar formula for crossed modules of groups works here and \(v_a \horizprod v_b\) can be expressed by \(v_{h_1} \cdot \alpha_{f_1} (v_{h_2})\).  Therefore, the naturality condition~\eqref{eq:weak_action_naturality_horizontal} is reduced to \eqref{eq:weak_action_crossed_module_3}.

Summing up, a weak action on \(\Csttwocat\) of the strict \(2\)\nb-groupoid~\(\cC\) associated to a crossed module \((G,H,\tcm,\acm)\) is given by \(\Cst\)\nb-algebras~\(A_x\), \Star{}isomorphisms \(\alpha_g\colon A_{\source(g)}\to A_{\target(g)}\), unitary multipliers \(\omega(g_1,g_2)\in \U\Mult(A_{\target(g_1)})\) for composable \(g_1,g_2\in G\), and \(\twc_h\in\U\Mult(A_x)\) for \(h\in H_x\), such that~\(\alpha\) is a groupoid homomorphism and equations \eqref{eq:weak_action_omega}, \eqref{eq:Busby_Smith_cocycle}, \eqref{eq:weak_action_crossed_module_1}, \eqref{eq:weak_action_crossed_module_2}, and \eqref{eq:weak_action_crossed_module_3} hold.  These rather complicated conditions describe the correct way to combine twists of group actions in the sense of Busby--Smith and Green.

There is also an equivalence between topological crossed modules and topological strict \(2\)\nb-groupoids.  We may add continuity requirements to the weak actions.  This works as in~\cite{Buss-Meyer-Zhu:Non-Hausdorff_symmetries} and in our discussion of weak actions on \(\Corrcat\) in Section~\ref{sec:topological_group}.  We require the spaces \(X\), \(G\), and~\(H\) to be locally compact, and we replace the set of \(\Cst\)\nb-algebras \((A_x)_{x\in X}\) by a \(\Cont_0(X)\)-\(\Cst\)\nb-algebra.  The continuity of~\(\alpha_g\) means that these automorphisms form the fibre restrictions of a \(\Cont_0(G)\)-linear \Star{}isomorphism \(\source^*(A)\to\target^*(A)\).  The continuity of \(\omega(g_1,g_2)\) and~\(\twc_h\) means that they form fibre restrictions of a unitary multiplier of the pull-back of~\(A\) to \(G\times_X G\) and~\(H\), respectively.  It makes no difference whether we require \(h\mapsto \twc_h\) or \((h,g)\mapsto \twc_{(h,g)}\) to be continuous.

Consider the crossed module of a closed normal subgroup \(N\triangleleft G\), so that its strict actions are exactly the twisted actions in Green's sense (see~\cite{Buss-Meyer-Zhu:Non-Hausdorff_symmetries}).  In this case, it is well-known that Green twisted actions of \(N\to G\) may be replaced by (measurable) Busby--Smith twisted actions of the quotient group~\(G/N\).  For a general crossed module, this is not possible because there is no good analogue of the quotient group~\(G/N\).  Conversely, we will see that after stabilisation the Busby--Smith twists can be removed, so that weak actions of a crossed module are stably equivalent to strict actions in a certain sense.

The reduction of Green twisted actions to Busby--Smith twisted actions is related to the Morita equivalence between the weak topological \(2\)\nb-categories \(N\to G\) and~\(G/N\), which \emph{should} induce an equivalence between the corresponding categories of weak actions.  This is indeed true for discrete weak \(2\)\nb-categories, but there are some technical problems with topologies here --~this is why the Busby--Smith twisted action of~\(G/N\) associated to a Green twisted action is, in general, only \emph{measurable}.  This problem appears because we used a simplified notion of weak action.  For topological \(2\)\nb-categories, our definitions involve continuous functors between topological groupoids.  As in more classical situations, these should be replaced by \emph{Hilsum--Skandalis morphisms} (see~\cite{Hilsum-Skandalis:Morphismes}) for best results. But since the definitions above already seem sufficiently complicated, we do not discuss this modification of our definitions any further here.

\subsection{Transformations}
\label{sec:transformations_two}

Now we extend the definition of transformations or weakly equivariant maps in Section~\ref{sec:weakly_equivariant_maps} from group actions to general weak actions of weak \(2\)\nb-groupoids.

Let~\(G\) and~\(\Cattwo\) be weak \(2\)\nb-categories and let \((A,\alpha,\omega_A,u_A)\) and \((B,\beta,\omega_B,u_B)\) be morphisms \(G\rightrightarrows\Cattwo\); we view \(\alpha\) and~\(\beta\) as functors, so that they also contain maps \(G_2\to\Cattwo_2\).

A \emph{transformation} between these two weak actions consists of the following data:
\begin{itemize}
\item arrows \(f_x\colon A_x \to B_x\) for \(x\in G_0\), and

\item bigons \(V_g\colon \beta_g f_{\source(g)} \Rightarrow f_{\target(g)}\alpha_g\) for \(g\in G_1\).
\end{itemize}
We require these bigons to be natural, that is, for each bigon \(n\colon g \Rightarrow h\), we require a commuting diagram
\begin{equation}
  \label{eq:transformation_weak_natural_1}
  \begin{gathered}
    \xymatrix@C+2em{
      \beta_g f_{\source(g)} \ar@{=>}[r]^-{V_g}
      \ar@{=>}[d]_{\beta_n}&
      f_{\target(g)}\alpha_g \ar@{=>}[d]^{f_{\target(g)}(\alpha_n)}\\
      \beta_h f_{\source(h)} \ar@{=>}[r]_-{V_h}&
      f_{\target(h)}\alpha_h.
    }
  \end{gathered}
\end{equation}
And we require exactly the same coherence laws \eqref{eq:weakly_equivariant_twist1} and~\eqref{eq:weakly_equivariant_twist2} as for actions of groups (but now we decorate the maps~\(f\) by appropriate indices indicating objects of~\(G\)).

If the bigons~\(V_g\) are invertible, the transformation is called \emph{strong}.  If all bigons~\(V_g\) are identity bigons, the transformation is called \emph{strict} or an \emph{equivariant map}.  If the transformation is strong with \(A=B\) and \(f_x=\Id_{A_x}\) for all \(x\in G_0\), then we also call it an (exterior) \emph{equivalence} between the two weak actions.

Now we consider the case where a strict \(2\)\nb-groupoid~\(\cC\) corresponding to a crossed module \((G,H,\tcm,\acm)\) acts weakly on \(\Csttwocat\).  Since all bigons in~\(\Csttwocat\) are invertible, any transformation to this category is strong.  It should be clear by now how to formulate continuity conditions in case \(X\), \(G\) and~\(H\) are locally compact.

\begin{remark}
  \label{rem:transformation_from_untaries}
  The construction in Definition~\ref{def:equivalence} also works for weak actions of weak \(2\)\nb-categories: given a weak action \((B,\beta,\omega,u)\) and invertible bigons \(V_g\colon \beta_g\Rightarrow \beta'_g\), there is a unique weak action \((B,\beta',\omega',u')\) such that \((\Id_B,V_g)\) is a strong transformation from \((B,\beta',\omega',u')\) to \((B,\beta,\omega,u)\).
\end{remark}

As a consequence, Lemmas \ref{lem:decompose_strong_strict_equivalence} and~\ref{lem:weaken_unit} remain true for weak actions of weak \(2\)\nb-categories, that is, any (continuous) transformation between (continuous) weak actions in \(\Csttwocat\) decomposes into a (continuous) strictly equivariant map and a (continuous) equivalence, and any (continuous) weak action in \(\Csttwocat\) is (continuously) equivalent to a (continuous) weak action with \(u_x=1\) for all \(x\in X\), so that \(\alpha_{1_x}=\Id_{A_x}\) and \(\omega(g,1_{\source(g)})=1\) and \(\omega(1_{\target(g)},g)=1\) for all \(g\in G\).  Recall that the bigons \(u_x\colon \Id_{A_x}\Rightarrow \alpha_{1_x}\) in a general weak action in \(\Csttwocat\) are redundant: \(u_x= \omega^*(1_x,1_x)\).

Let \((A, \alpha, \omega_A,\twc^A)\) and \((B, \beta, \omega_B,\twc^B)\) be the data describing two weak actions of~\(\cC\) as in Section~\ref{sec:morphisms_two}.  A transformation between them consists of non-degenerate \Star{}homomorphisms \(f_x\colon A_x \to \Mult(B_x)\) for \(x\in G_0\) and unitary multipliers \(V_g \in \U\Mult(B_{\target(g)})\) for \(g\in G_1\), satisfying~\eqref{eq:transformation_intertwine} and~\eqref{eq:transformation_cocycle} --~these are the conditions for transformations between Busby--Smith twisted actions~-- and the additional naturality condition
\begin{equation}
  \label{eq:naturality-csttwounit}
  V_{\tcm(h)} \cdot \twc^B_h = f(\twc^A_h) \cdot V_1,
\end{equation}
for all \(h\in H\); this implies the stronger naturality condition \(V_{\tcm(h)g} \cdot \twc^B_{(h,g)} = f(\twc^A_{(h,g)}) \cdot V_g\) for all \((h,g)\in H\times_X G\), which is equivalent to~\eqref{eq:transformation_weak_natural_1}.  Continuity means that the maps~\(f_x\) are the fibre restrictions of a non-degenerate \Star{}homomorphism \(f\colon A\to \Mult(B)\) and the unitary multipliers~\(V_g\) are the fibre restrictions of a unitary multiplier~\(V\) of the pull-back of~\(B\) to~\(G\).

If our crossed module is just a group, so that we are dealing with Busby--Smith twisted actions, then we get exactly the same notion of transformation as in Section~\ref{sec:weakly_equivariant_maps}.

If our actions are strict, that is, \(\omega_A=1\) and \(\omega_B=1\), then \(f\) and~\((V_g)\) form a transformation if and only if they satisfy \eqref{eq:transformation_intertwine}, \eqref{eq:naturality-csttwounit}, and the condition
\[
V_{g_1g_2} = V_{g_1} \cdot \beta_{g_1}(V_{g_2})
\]
for all \(g_1,g_2\in G\), which comes from~\eqref{eq:transformation_cocycle}.  If we also require~\(f\) to be invertible, this gives us exactly the definition of exterior equivalence of Green's twisted action in~\cite{Echterhoff:Morita_twisted}.

\subsection{Modifications of transformations}
\label{sec:equiv_weak_equivariant_gpd}

Now we weaken equality of transformations to the notion of a modification as in Section~\ref{sec:equiv_weak_equivariant}.  Let \((A,\alpha,\omega_A,u_A)\) and \((B,\beta,\omega_B,u_B)\) be two morphisms from~\(G\) to a \(2\)\nb-category~\(\Cattwo\) and let \((f,V)\) and \((f',V')\) be transformations between them.  A \emph{modification} \((f,V) \Rightarrow (f',V')\) consists of bigons \(W_x\colon f_x\Rightarrow f'_x\) for all \(x\in G_0\) that satisfy the coherence law
\begin{equation}
  \label{eq:modification_gpd_action}
  \begin{gathered}
    \xymatrix{
      \beta_g f_y \ar@{=>}[r]^{V_g}
      \ar@{=>}[d]_{\beta_g\horizprod W}&
      f_x \alpha_g \ar@{=>}[d]^{W\horizprod \alpha_g}\\
      \beta_g f'_y \ar@{=>}[r]^{V'_g}&
      f'_x\alpha_g.
    }
  \end{gathered}
\end{equation}
analogous to~\eqref{eq:modification_group_action} (see~\cite{Leinster:Basic_Bicategories}).  Notice that this naturality condition does not involve the bigons of~\(G\).

An invertible modification \(W\colon (f,V)\Rightarrow (f',V')\) (that is, the bigon~\(W\) is invertible) is also called an \emph{equivalence}, and then \((f,V)\) and~\((f',V')\) are called \emph{equivalent}.

Now we specialise this to the category~\(\Csttwocat\).  Let \((f,V)\) and \((f',V')\) be transformations from \((A,\alpha,\omega_A,u_A,v^A)\) to \((B,\beta,\omega_B,u_A,v^A)\).  Since any bigon in~\(\Csttwocat\) is invertible, any modification is an equivalence.  A modification \((f,V)\Rightarrow (f',V')\) consists of unitary multipliers~\(W_x\) of~\(B_x\) that satisfy
\begin{equation}
  \label{eq:modification_Cstar1}
  f'_x(a) = W_x\cdot f_x(a)\cdot W_x^*
  \qquad\text{and}\qquad
  V'_g = W_x\cdot V_g\cdot \beta_g(W_y^*)
\end{equation}
for all \(a\in A_x\), \(g\in G_1\) with \(\target(g)=x\) and \(\source(g)=y\).  Continuity means that the unitary multipliers~\(W_x\) are the fibre restrictions of a multiplier of~\(A\).

\section{Stabilisation results}
\label{sec:stabilisation}

In this section, we use known results about Hilbert modules and about Busby--Smith twisted actions to show that, after stabilisation, all \emph{weak} actions of \emph{strict} \(2\)\nb-categories by correspondences become equivalent to strict actions.  The starting point is the following well-known Triviality Theorem:

\begin{theorem}
  \label{the:triviality_theorem}
  Let \(A\) and~\(B\) be \(\sigma\)\nb-unital \(\Cst\)\nb-algebras and let~\(\Hilm\) be a correspondence from~\(A\) to~\(B\).  If~\(A\) is stable and~\(\Hilm\) is countably generated and full as a Hilbert \(B\)\nb-module, then the underlying Hilbert \(B\)\nb-module of~\(\Hilm\) is unitarily isomorphic to \(B^\infty \defeq B\otimes \ell^2(\N)\).
\end{theorem}

\begin{proof}
  Since~\(A\) is stable, \(A\cong \Comp(\ell^2\N) \otimes A_0\) (with \(A\cong A_0\)).  The \Star{}representation of~\(A\) on~\(\Hilm\) extends to a strictly continuous unital \Star{}homomorphism on the multiplier algebra, so that~\(\Hilm\) carries a \Star{}representation \(\pi\colon \Comp(\ell^2\N)\to\Bound(\Hilm)\).  Let~\(e_{ij}\) for \(i,j\in\nobreak\N\) be the matrix units in \(\Comp(\ell^2\N)\).  The operators \(\pi(e_{ii})\) in \(\Bound(\Hilm)\) are orthogonal projections with strict convergence \(\sum_{i\in\N} \pi(e_{ii}) = 1\).  Letting \(\Hilm_i\subseteq\Hilm\) be the range of~\(\pi(e_{ii})\), we get a direct sum decomposition \(\Hilm\cong \bigoplus_{i\in \N} \Hilm_i\).  The operator~\(\pi(e_{ij})\) restricts to a unitary operator between~\(\Hilm_j\) and~\(\Hilm_i\), so that we may identify each summand~\(\Hilm_j\) with~\(\Hilm_0\).  Thus \(\Hilm\cong \Hilm_0^\infty\).  Since~\(\Hilm\) is full, so is~\(\Hilm_0\).  Finally, \cite{Mingo-Phillips:Triviality}*{Theorem 1.9} yields \(\Hilm_0^\infty \cong B^\infty\) because~\(\Hilm_0\) is countably generated and full.
\end{proof}

Next we formulate the Packer--Raeburn Stabilisation Trick in our context (compare \cite{Kaliszewski:Morita_twisted}*{Theorem 2.1}).  Let \(A_\Comp\defeq A\otimes\Comp(\ell^2\N)\) denote the \(\Cst\)\nb-stabilisation of a \(\Cst\)\nb-algebra~\(A\).  A continuous weak or strong action of a weak \(2\)\nb-category on~\(A\) by \Star{}automorphisms or correspondences induces a continuous action of the same kind on its stabilisation~\(A_\Comp\): simply map all \Star{}homomorphisms~\(\alpha\) in the data to \(\alpha_\Comp\defeq \alpha\otimes\Id_\Comp\) and all unitaries~\(u\) in the data to \(u_\Comp\defeq u\otimes 1\).

\begin{proposition}[Packer--Raeburn Stabilisation Trick]
  \label{pro:stable_action_trivialise}
  Let~\(G\) be a second countable locally compact groupoid with Haar system, let~\(X\) be its object space.  Let \((B,\beta,\omega)\) be a continuous weak action of~\(G\) on a \(\Cst\)\nb-algebra~\(B\) over~\(X\).  Then the weak action \((B_\Comp,\beta_\Comp,\omega_\Comp)\) is equivalent to a strict action, that is, to an action of the form \((\beta',\omega')\) with \(\omega'_{g,h}=1\) for all \(g,h\in G\).
\end{proposition}

\begin{proof}
  Let~\((\lambda^x)_{x\in X}\) be a left invariant Haar system on~\(G\) and let \(L^2(G)\) be the continuous field of Hilbert spaces with fibres \(L^2(G^x,\lambda^x)\).  In our assumption, we use compact operators on the Hilbert space \(\ell^2\N\).  But we may as well use compact operators on the continuous field of Hilbert spaces \(L^2(G)\otimes\ell^2\N\) because the latter is isomorphic to a constant continuous field \(\ell^2\N\) --~this is actually a special case of Theorem~\ref{the:triviality_theorem}.  Since we will not use the tensor factor \(\Comp(\ell^2\N)\) in the following, we subsume it in~\(B\) and assume from now on that we are dealing with a weak action of the form \((\beta\otimes\Id_\Comp,\omega\otimes1)\) on \(B\otimes \Comp\bigl(L^2(G)\bigr)\).

  We want to use the construction in Definition~\ref{def:equivalence} to construct an equivalence \((V_g)_{g\in G}\) between \((\beta\otimes\Id_\Comp,\omega\otimes1)\) and a strict action, that is, a weak action of the form \((\beta',\omega')\) with \(\omega'\equiv1\); here \(V_g\colon \beta_g\Rightarrow\beta'_g\).  Equation~\ref{eq:horizprod_Cstartwocat} yields \(V_{g_1}^{-1}\horizprod V_{g_2}^{-1} =  (\beta\otimes\Id_\Comp)_{g_1}(V_{g_2}^{-1})\cdot V_{g_1}^{-1}\), so that
  \[
  \omega'(g_1,g_2) = V_{g_1g_2}\cdot (\omega\otimes1)(g_1,g_2) \cdot (\beta\otimes\Id_\Comp)_{g_1}(V_{g_2}^{-1}) V_{g_1}^{-1}.
  \]
  Hence we need~\(V_g\) to satisfy \(V_{g_1g_2}\cdot (\omega\otimes1)(g_1,g_2) = V_{g_1}\cdot (\beta\otimes\Id_\Comp)_{g_1}(V_{g_2})\).  We also want the equivalence to be continuous, that is, \(V=(V_g)_{g\in G}\) should be a unitary operator on the pull-back of \(L^2(G)\otimes_X B\) along the range map \(\target\colon G\to X\).

  Combining pointwise multiplication by \(\omega(g,h)\) and the left regular representation of~\(G\), we define
  \[
  (V_g f)(h) \defeq \omega(g,g^{-1}h) f(g^{-1}h)
  \qquad\text{for all \(f\in \Cont_c(G,B)\), \(g\in G\), \(h\in G^{\target(g)}\),}
  \]
  where \(f(g)\in B_{\target(g)}\).  Note that \((\beta\otimes\Id_\Comp)_{g_1}(V_{g_2}) f(h)= \beta_{g_1}\bigl(\omega(g_2,g_2^{-1}h)\bigr) f(g_2^{-1}h)\) and hence
  \begin{align*}
    V_{g_1}\cdot (\beta\otimes\Id_\Comp)_{g_1}(V_{g_2}) f(h)
    &= \omega(g_1,g_1^{-1}h) (\beta\otimes\Id_\Comp)_{g_1}(V_{g_2})
    f(g_1^{-1}h)\\
    &= \omega(g_1,g_1^{-1}h)
    \beta_{g_1}\bigl(\omega(g_2,g_2^{-1}g_1^{-1}h)\bigr)
    f(g_2^{-1}g_1^{-1}h),\\
    (V_{g_1g_2}\cdot (\omega\otimes1)(g_1,g_2) f)(h)
    &= \omega(g_1g_2,(g_1g_2)^{-1}h)\bigl((\omega\otimes1)(g_1,g_2)
    f\bigr) \bigl((g_1g_2)^{-1}h\bigr)\\
    &= \omega(g_1g_2,g_2^{-1}g_1^{-1}h) \omega(g_1,g_2)
    f(g_2^{-1}g_1^{-1}h).
  \end{align*}
  These two are equal by the groupoid generalisation of the cocycle condition~\eqref{eq:Busby_Smith_cocycle} for~\(\omega\), applied to \((g_1,g_2,g_2^{-1}g_1^{-1}h)\).  Therefore, when we use the groupoid generalisations of \eqref{eq:outer_equivalence_intertwine} and~\eqref{eq:outer_equivalence_cocycle} to construct an equivalent weak action, we get \(\beta'_g(b) = V_g (\beta\otimes\Id)_g(b) V_g^*\) for all \(b\in B\otimes\Comp\), \(g\in G\) and \(\omega'(g_1,g_2)=1\) for all \(g_1,g_2\).  Thus~\((\beta',1)\) is a strict continuous action that is continuously equivalent to \((\beta,\omega)\).
\end{proof}

\begin{theorem}
  \label{the:simplify_correspondence_weak}
  Let \((G,H,\tcm,\acm)\) be a crossed module of second countable, locally compact groupoids with object space~\(X\).  Let~\(A\) be a \(\sigma\)\nb-unital stable \(\Cont_0(X)\)-\(\Cst\)\nb-algebra.  Then any continuous action \((\alpha,\omega,u)\) by correspondences of \((G,H,\tcm,\acm)\) on~\(A\) is equivalent to a strict action, that is, there are a continuous strict action of the crossed module \((G,H,\tcm,\acm)\) on~\(A\) in the sense of~\textup{\cite{Buss-Meyer-Zhu:Non-Hausdorff_symmetries}} and an invertible continuous transformation from~\(A\) with the action \((\alpha,\omega,u)\) to~\(A\) with this strict action.
\end{theorem}

\begin{proof}
  The proof proceeds in two steps.  First, we use Theorem~\ref{the:triviality_theorem} to construct an equivalence to a weak action by \Star{}automorphisms, that is, a weak action in \(\Csttwocat\).  Then we use the Packer--Raeburn Stabilisation Trick (Proposition~\ref{pro:stable_action_trivialise}) to replace the latter by a strict action.

  Let~\(\cC\) be the weak \(2\)\nb-category associated to the crossed module \((G,H,\tcm,\acm)\) and take a weak action of~\(\cC\) on~\(A\) by correspondences, that is, a continuous morphism \(\cC\to\Corrcat\) that is given on objects by the \(\Cont_0(X)\)-\(\Cst\)\nb-algebra~\(A\).  On the arrows, it is given by a \(\Cont_0(G)\)\nb-linear invertible correspondence~\(\alpha\) from \(\Cont_0(G,A)\) to itself.  For the time being, we concentrate on~\(\alpha\) and may ignore the remaining data.

  Since~\(\alpha\) is an imprimitivity bimodule, it is full as a right Hilbert \(\Cont_0(G,A)\)-module.  The \(\Cst\)\nb-algebra \(\Cont_0(G,A)\) is \(\sigma\)\nb-unital and stable because~\(A\) is \(\sigma\)\nb-unital and~\(G\) is \(\sigma\)\nb-compact.  Hence~\(\alpha\) is countably generated because \(\Comp(\alpha) \cong \Cont_0(G,A)\) is \(\sigma\)\nb-unital (see \cite{Mingo-Phillips:Triviality}*{Corollary 1.5}).  Theorem~\ref{the:triviality_theorem} implies \(\alpha \cong \Cont_0(G,A)^\infty \cong \Cont_0(G,A)\) because~\(A\) is stable.

  This unitary operator provides a strong transformation from the given action \((\alpha,\omega,u,\twc)\) to another weak action by correspondences \((\alpha'',\omega'',u'',\twc'')\) such that \(\alpha''=\Cont_0(G,A)\) as a right Hilbert \(\Cont_0(G,A)\)-module.  The left module structure on~\(\alpha''\) provides a \(\Cont_0(G)\)-linear \Star{}automorphism on \(\Cont_0(G,A)\).  By restriction to the fibres, this corresponds to a strongly continuous map \(\alpha'\colon G\to\Aut(A)\).  By construction, \(\alpha''=[\alpha']\) in the notation of Example~\ref{exa:automorphism_Hilbert_bimodule}.

  Multipliers of \(\Cont_0(G\times G,A)\) are strictly continuous maps from~\(G\times G\) to the multiplier algebra of~\(A\).  Hence we may view the unitary operators~\(\omega''\) on the Hilbert module \(\Cont_0(G\times G,A)\) as a strictly continuous family of unitary multipliers \(\omega'(g,h)\).  Similarly, we view \(u''\) and~\(\twc''\) as unitary multipliers \(u'\) and~\(\twc'\) of \(A\) and the pull-back of~\(A\) to~\(H\).  The resulting triple \((\alpha',\omega',u',\twc')\) is a weak action of~\(\cC\) by \Star{}automorphisms, that is, a morphism \(\cC\to\Csttwocat\).  We have completed the first step and shown that weak actions by correspondences on stable \(\Cst\)\nb-algebras are equivalent to weak actions by \Star{}automorphisms.

  In the second step, we let \((\alpha,\omega,u,\twc)\) be a weak action of~\(\cC\) by \Star{}automorphisms, which we continue to view as a weak action by correspondences.  The standard Morita equivalence \(f\defeq \ell^2(\N,A)\) between \(A\) and~\(A_\Comp\) provides an equivalence between \((\alpha,\omega,u,\twc)\) and \((\alpha_\Comp,\omega_\Comp,u_\Comp,\twc_\Comp)\).  More precisely, we must combine~\(f\) with the canonical isomorphisms
  \[
  V_g\colon \ell^2(\N,A)\otimes_A [\alpha_g] \congto
  \ell^2(\N)\otimes [\alpha_g] \congto
  [\alpha_g\otimes\Id_\Comp] \otimes_{A_\Comp} \ell^2(\N,A).
  \]
  Proposition~\ref{pro:stable_action_trivialise} shows that the weak action \((\alpha_\Comp,\omega_\Comp,u_\Comp)\) of~\(G\) is equivalent to a strict action~\(\alpha'\), that is, a weak action \((\alpha',\omega',u')\) with \(\omega'=1\) and \(u'=1\).  Now Remark~\ref{rem:transformation_from_untaries} shows that the invertible bigons that implement this equivalence also provide an equivalence of weak actions of the crossed module \((G,H)\); that is, we get a unique homomorphism \(\twc'\colon H\to \U\Mult(A_\Comp)\) such that \((\alpha',\omega,u',\twc')\) is a weak action equivalent to \((\alpha_\Comp,\omega_\Comp,u_\Comp,v_\Comp)\).  Composing the two equivalences above, we get a weak equivalence \((\Phi,W)\) between the action \((\alpha,\omega,u,\twc)\) of~\(\cC\) and a weak action \((\alpha',\omega,u',\twc')\) with \(u'=1\) and \(\omega'=1\).  The latter is nothing but a (strict) action of the crossed module~\(\cC\) in the usual sense.

  Since the equivalence \(A_\Comp\simeq A\) is only a correspondence, \(\Phi\) is only a correspondence as well.  But since~\(\Phi\) is an invertible correspondence between two stable \(\Cst\)\nb-algebras, it is of the form~\([\varphi]\) for a \Star{}isomorphism \(\varphi\colon A\to A_\Comp\).  Then \(\alpha''_g(a) \defeq \varphi\alpha_g(a)\varphi^{-1}\) for \(a\in A\) and \(\twc''_h \defeq \varphi\cdot\twc\) defines a continuous strict action of the crossed module on~\(A\) that is equivalent to the original action.
\end{proof}

\begin{corollary}
  If \((\alpha,\omega,u,v)\) is a continuous action by correspondences of a second countable, locally compact crossed module \((G,H,\tcm,\acm)\) with object space~\(X\) on a \(\sigma\)\nb-unital \(\Cont_0(X)\)-\(\Cst\)-algebra~\(A\), then the induced action on its stabilisation~\(A_\Comp\) is equivalent to a continuous strict action of \((G,H,\tcm,\acm)\) by \Star{}homomorphisms and unitary intertwiners.
\end{corollary}

By Theorem~\ref{the:group_act_Corrcat_continuous}, continuous actions of~\(G\) by correspondences correspond bijectively to saturated Fell bundles over~\(G\).  Moreover, invertible continuous transformations between weak actions by correspondences correspond bijectively to Morita equivalences of the associated Fell bundles.  Thus we immediately get the following consequence:

\begin{corollary}
  \label{cor:stable_Fell_bundle=ordinary_action}
  If~\(G\) is a locally compact, second countable locally compact group, then any saturated Fell bundle over~\(G\) for which the unit fiber is a \(\sigma\)\nb-unital stable \(\Cst\)\nb-algebra is Morita equivalent to a Fell bundle over~\(G\) that is associated to a strict \(G\)\nb-action.  Hence every Fell bundle is, up to stabilisation, Morita equivalent to one associated to a strict group action.
\end{corollary}

\begin{remark}
  \label{rem:strictify_Fell_alternative}
  The above result is well-known to the specialists.  It can also be proved by combining the Packer--Raeburn Stabilisation Trick in~\cite{Kaliszewski:Morita_twisted} and the classification result of Ruy Exel for Fell bundles in~\cite{Exel:TwistedPartialActions}.  Indeed, the main result in~\cite{Exel:TwistedPartialActions} says that a Fell bundle~\(A\) over~\(G\) which is \emph{regular} --~a technical condition that need not concern us here~-- is isomorphic to a Fell bundle coming from a (Busby--Smith) \emph{twisted partial action} of~\(G\).  Twisted partial actions are natural generalisations of Busby--Smith twisted actions where one allows partially defined isomorphisms between ideals of a given \(\Cst\)\nb-algebra.  Following Exel's construction in~\cite{Exel:TwistedPartialActions}, we get a twisted \emph{global} action if the initial regular Fell bundle is saturated.  Thus Exel's main result in~\cite{Exel:TwistedPartialActions} implies that regular, saturated Fell bundles over~\(G\) correspond to Busby--Smith twisted (global) actions of~\(G\).  And by the Packer--Raeburn Stabilisation Trick, every twisted \(G\)\nb-action is Morita equivalent to an ordinary \(G\)\nb-action.  Moreover, as observed by Exel in~\cite{Exel:TwistedPartialActions}, Fell bundles with stable unit fiber are automatically regular.  All this together yields an alternative proof of Corollary~\ref{cor:stable_Fell_bundle=ordinary_action}.
\end{remark}

\section{Conclusion}
\label{sec:conclusion}

We have defined weak actions of weak \(2\)\nb-categories on \(\Cst\)\nb-algebras, transformations between such actions, and modifications of such transformations.  Our starting point was the notationally simpler case of group actions, where our notions specialise to known concepts like Busby--Smith twisted actions, equivalence of such actions, covariant representations, and unitary intertwiners between covariant representations.  Weak actions of strict \(2\)\nb-groupoids generalise Busby--Smith twisted actions and Green twisted actions at the same time.  Furthermore, when we use the correspondence category of \(\Cst\)\nb-algebras, we get notions related to saturated Fell bundles.

Finally, we used results about Morita equivalence of \(\Cst\)\nb-algebras to strictify weak actions after stabilisation.  This strictification is a rather unusual feature of \(\Cst\)\nb-algebras.  In most contexts, weak actions of \(2\)\nb-groupoids are far from equivalent to strict actions.

\end{document}